\documentclass[proc]{edpsmath}
\usepackage{amssymb,amsthm,amsmath}
\usepackage{hyperref}
\usepackage[utf8]{inputenc}
\usepackage[english]{babel}
\usepackage{amsfonts}
\usepackage{blindtext}
\usepackage[font=small]{caption}
\usepackage{subcaption}
\usepackage{physics}
\usepackage{bm}
\usepackage{tabularx}
\usepackage{graphicx}
\usepackage{mathtools}
\usepackage{caption} 
\usepackage{float}
\usepackage{multicol}
\usepackage{wrapfig}
\usepackage{subfiles}
\newtheorem{theorem}{Theorem}[section]
\theoremstyle{definition}

\newtheorem{remark}[theorem]{Remark}
\newtheorem{Example}[theorem]{Example}

\usepackage{pgfplots}
\usepgfplotslibrary{groupplots,dateplot}

\numberwithin{equation}{section} 
\title{A Note on Hyperbolic Relaxation of the Navier-Stokes-Cahn-Hilliard system for incompressible two-phase flow}
\usepackage{epstopdf}
\usepackage{stmaryrd} 
\usepackage{fixmath} 
\usepackage{pgf}
\usepackage{pgfplots}
\usepackage{wasysym}

\newcommand{\eps}{\varepsilon}
\newcommand{\ds}{\displaystyle}

\newcommand{\vecstyle}[1]{{\bf{#1}}}
\newcommand{\vecu}{\vecstyle{u}}
\newcommand{\vece}{\vecstyle{e}}

\newcommand{\vecd}{\vecstyle{d}}

\newcommand{\vecx}{\vecstyle{x}}

\newcommand{\vecn}{\vecstyle{n}} 

\newcommand{\vecf}{\vecstyle{f}}
\newcommand{\vecv}{\vecstyle{v}}
\newcommand{\setstyle}[1]{{\mathbb #1}}

\def\setR {\setstyle{R}}

\begin{document}

%
\author{Jens Keim}\address{Institute of Aerodynamics and Gas Dynamics, University of Stuttgart, Wankelstra{\ss}e 3,  70563 Stuttgart,  Germany\\
Email: {\tt{jens.keim@iag.uni-stuttgart.de}}
}
\author{Hasel-Cicek Konan}\address{Institute of Applied Analysis and Numerical Simulation, University of Stuttgart, Pfaffenwaldring 57,  70569 Stuttgart,
Germany.
Email: {\tt{Hasel-Cicek.Konan@mathematik.uni-stuttgart.de}}
}
\author{Christian Rohde}\address{
Institute  of Applied Analysis and Numerical Simulation, University of Stuttgart, Pfaffenwaldring 57,  70569 Stuttgart,
Germany.
Email: {\tt{christian.rohde@mathematik.uni-stuttgart.de}}
}

%
%
\begin{abstract} 
We consider the two-phase dynamics of two incompressible and immiscible fluids. As a mathematical model we rely on the Navier-Stokes-Cahn-Hilliard  system that belongs to the class of diffuse-interface models. Solutions of the  Navier-Stokes-Cahn-Hilliard system exhibit strong non-local effects due to the velocity divergence constraint and the fourth-order Cahn-Hilliard operator. We suggest a new first-order approximative system for the inviscid sub-system. It relies on the artificial-compressibility ansatz for the Navier-Stokes equations, a friction-type approximation for the Cahn-Hilliard equation and a relaxation of a third-order capillarity term. We show under reasonable assumptions that the first-order operator within the approximative system is hyperbolic; precisely we prove for the spatially one-dimensional case that it is equipped with an entropy-entropy flux pair with convex (mathematical) entropy. For specific states we present a numerical characteristic analysis. Thanks to the hyperbolicity of the system, we can employ all standard numerical methods from the field of hyperbolic conservation laws. We conclude the paper with preliminary numerical results in one spatial dimension.\\
\textit{Keywords: two-phase flow, diffuse-interface modelling, Navier-Stokes-Cahn-Hilliard system, first-order hyperbolic relaxation}
\end{abstract}
%
%
\maketitle
\section{Introduction} 
The incompressible motion of a viscous  one-phase  fluid  is governed by the Navier--Stokes system. In absence of external forces and setting the density of the fluid to be $1$ it is given in  $d\in \{2,3\}$ space dimensions
by the equations 
\begin{equation}\label{NS}
    \begin{array}{c}
         \begin{array}{rcccl}
             {} &~& \div \vecu &=& 0,    \\[1.5ex]
             \vecu_{t} &+&  (\vecu \cdot \grad) \vecu+ \grad p  &=& \nu \Delta \vecu
         \end{array}
         \quad \text{ in }  \Omega \times (0,T).
    \end{array}
\end{equation}
Here $\Omega \subset \setR^d$ is an open set and $T>0$ some end time. By $\nu >0$ we denote the viscosity parameter. 
The unknowns are the velocity $\vecu= \vecu(\vecx,t) \in \setR^d$ and the pressure (perturbation)  $p=p(\vecx,t) \in \setR$. Classical solutions of \eqref{NS} dissipate the kinetic energy.\\
The divergence constraint $\eqref{NS}_1$ leads to  a strong non-locality for solutions of the  Navier--Stokes system such that the pressure is governed by an elliptic differential equation.  
A classical approach to avoid this non-locality is the artificial-compressibility approximation, see 
\cite{chorin1967ams,chorin1967jcp,chorin1968} and \cite{temam1969I,temam1969II}. In this case, 
the Navier--Stokes system is approximated by the set of evolution equations 
\begin{equation}\label{NSapproximation}
\begin{array}{c}
\begin{array}{rcccl}
  \alpha p_{t}^{\alpha} &+& \ds  \divergence \vecu^{\alpha} &= &0, \\[2.0ex]
  \vecu_{t}^{\alpha} &+& \ds 
   (\vecu^\alpha \cdot \grad) \vecu^\alpha + \frac12(\grad\cdot \vecu^\alpha)\vecu^\alpha   
  + \grad p^{\alpha}  &=&  \nu \Delta \vecu^{\alpha}
\end{array} \quad \text{ in }  \Omega \times (0,T).
\end{array}
\end{equation} 
Adding a time derivative for the pressure and scaling it with a small 
parameter $\alpha >0$, the pressure is governed now  by an evolution equation. For $\alpha \to 0 $ solutions    $\vecu^\alpha(\vecx,t)$ and  $p^\alpha = p^\alpha(\vecx,t)$ of the system 
\eqref{NSapproximation}  recover the solutions   $\vecu$  and $p$  of the original Navier--Stokes system \eqref{NS}. The additional correction
term $\frac12(\grad\cdot \vecu^\alpha)\vecu^\alpha $ in \eqref{NSapproximation}$_2$ ensures that classical solutions dissipate the sum of kinetic energy and the potential contribution $\alpha (p^\alpha)^2/2$.   Most important,  the structure of the  artificial compressibility approximation allows to use   methods from the field of hyperbolic evolution equations  to understand the  first-order sub-system (obtained by setting $\nu$ to be zero)  in \eqref{NSapproximation}. This has in particular consequences for the numerical simulation.  One can  employ numerical methods for systems  of hyperbolic evolution laws to discretize  this  sub-system of \eqref{NSapproximation}. In such a way also convection-dominated flow regimes can be handled on the numerical level, see e.g.~\cite{gresho1991,nithiarasu2003,massa2022}.\\
In this note, we focus on the dynamics of incompressible  two-phase flows and aim at developing 
a framework that transfers the artificial-compressibility approach into this realm. We consider as a model problem a simplified version of the Navier--Stokes--Cahn--Hilliard (NSCH)  equations as  originally proposed in \cite{hohenberg1977theory}.  The NSCH model couples the Navier--Stokes equations with the fourth-order convective Cahn--Hilliard equation 
for the phase dynamics. Thus, it belongs to the class of diffuse-interface 
(or phase-field) approaches, see \cite{anderson1998diffuse}. 
The major challenge to approximate the  NSCH system using a first-order system is therefore the treatment of the coupled  convective Cahn--Hilliard 
equation. We refer to \cite{dhaouadi2024,Graselli} for approximation techniques for the pure Cahn--Hilliard evolution. 
In this work, we rely on a friction-type approximation that mimics the structure of Euler-Korteweg systems which have been considered as approximations of the
Cahn--Hilliard equations in, e.g.,~\cite{lattanzio17,MR4182904}. Roughly speaking, the fourth-order operator in the Cahn--Hilliard equation is substituted by a third-order Korteweg term. The last step to obtain a 
first-order evolution system (always ignoring the viscous part of the stress tensor which is not touched) is the use of the relaxation technique from 
\cite{rohde2010local}. We present our new approximative system with a first-order sub-system in \eqref{firstorderNSEK}.
\\
The remainder of the paper is structured as follows: In Section \ref{sec_NSCH}, we give a short outline of  the NSCH system and its thermodynamical structure.
In Section \ref{sec_NSEK}, we propose the friction-type approximation of the NSCH system and we show that it is a thermodynamically consistent system of evolution equations. 
In Section \ref{sec_formNSEK}, we introduce our new approximative system and show that it obeys a natural energy dissipation rate (see Theorem \ref{EDrelaxedNSEK}). Moreover, we show under fair conditions in one spatial dimension that 
it involves a hyperbolic first-order sub-system similar 
to the one in the artifical-compressibility approximation \eqref{NSapproximation}. For the one-dimensional sub-system, 
we present a numerical characteristic analysis
to understand the qualitative behaviour of the eigenvalues of the flux Jacobian. 
In Section \ref{sec_numer.experiments}, we present extensive numerical results for the one-dimensional version, i.e., the system \eqref{rewrittenfirstorderNSEK} below.

\section{The Navier-Stokes-Cahn-Hilliard system \label{sec_NSCH}}
In this section, we review a simplified version of the Navier--Stokes--Cahn--Hilliard (NSCH) system as originally introduced in \cite{hohenberg1977theory}. In view of the  approximations which we will introduce in the sequel, we recall the energy dissipation equation for solutions of the initial boundary problem for the NSCH system. Finally, we mention a result on the existence of strong solutions in two space dimensions.\\
The domain of the flow is the bounded open set $\Omega \subset \mathbb{R}^{d}$, $d \in \{2,3\} $, 
with boundary $\partial \Omega $. The outer normal of $\partial \Omega$ will be denoted by 
$\vecn \in {\mathcal S}^{d-1}$. For some end time  $T > 0$, we 
define the space-time domain 
$\Omega_{T} = \Omega \times (0,T)$.

\subsection{Formulation of the NSCH system \label{subsec_NSCH}}
Let us consider the dynamics of two incompressible and immiscible fluids at constant temperature.
There are two approaches to model such two-phase flows: the sharp-interface approach and the diffuse-interface approach. In the sharp-interface ansatz, one obtains a free-boundary value problem 
which separates the two bulk domains by a well-defined phase boundary as codimension-$1$ manifold. 
We focus here on the diffuse-interface (or phase-field) approach, where the phase boundary is smeared out with an interfacial width $\gamma > 0$, see Figure \ref{figure_diffuse}.
\begin{figure}[h!]
       \begin{center}
            \includegraphics[width=0.35\textwidth]{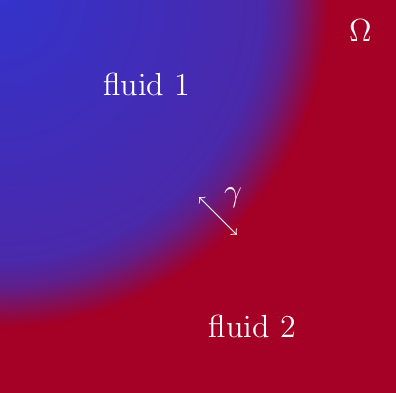}
             \caption{Sketch of a diffuse-interface configuration with phase boundary of width $\gamma$.  \label{figure_diffuse}}
        \end{center}
\end{figure}
If the parameter $\gamma$ tends to zero, one expects that an appropriate  sharp-interface model is attained in the limit. Here, we are not interested in the sharp-interface limit and therefore 
fix the parameter $\gamma$ throughout the entire note.\\
As a meanwhile well-known diffuse-interface approach, we start from the classical Navier--Stokes--Cahn--Hilliard (NSCH) system, which can be traced back to ~\cite{hohenberg1977theory}. 
It is a combination of the incompressible Navier--Stokes equations for the hydromechanical dynamics and the Cahn--Hilliard equation for the two-phase dynamics. The governing partial differential equations are given by 
\begin{equation}\label{NSCH}
    \begin{array}{c}
         \begin{array}{rcccl}
             {} &~& \div \vecu &=& 0,    \\[1.5ex]
             \vecu_{t} &+&  (\vecu \cdot \grad) \vecu+ \grad p  &=& -c \grad \mu(c) +\nu \Delta \vecu,  \\[1.5ex]
             c_{t} &+& \div (c\vecu) -  \div (\grad \mu(c)) &=&0, \\[1.5ex]
             {} &~& \mu(c) &=& W'(c)-\gamma \Delta c
         \end{array}
         \quad \text{ in }  \Omega_{T}.
    \end{array}
\end{equation}

The hydromechanical unknowns in \eqref{NSCH} are the pressure (perturbation) $p=p(\bm{x},t)\in \setR$ and the velocity field $\bm{u} = \bm{u}(\bm{x},t)\in \setR^d$.
Furthermore, we search for the phase-field variable 
$c=c(\bm{x},t)\in [-1,1]$ which acts as phase indicator (see Figure \ref{logarithmicvspolynomial}). For later use, we define the state vector  ${\bm U} = (p, \vecu^T,c)^T \in \mathcal U$  with the 
state space  $\mathcal U$ given by \[
{\mathcal U}= \left\{ (p,\vecu^T,c)^T \in \setR \times \setR^d\times \setR \, \left. \right| \,  c \in [-1,1]  \right\}.\] 
We denote the given fixed viscosity parameter by $\nu >0$, implicitly assuming that the viscosity is independent of the phase-field variable $c$. Moreover, we assumed that the densities of the two fluids are the same and equal to $1$.
The quantity $\mu(c)$  in \eqref{NSCH}  
represents the chemical potential which involves the free energy function $W:[-1,1] \to \setR$. Note that we consider here a special version of the 
Cahn--Hilliard equation where the mobility is supposed to be constant (equal to $1$).\\ 
Concerning the free energy function $W$, we choose in this paper the quartic polynomial
\begin{equation}\label{doublewell} 
{W} (c) = \dfrac{1}{4} (c^2-1)^2 \quad \forall \; c \in [-1,1],
\end{equation}
which has a double-well structure with two minima in the pure phase states, see Figure \ref{logarithmicvspolynomial}.  A
physically more grounded  alternative is the Flory-Huggins logarithmic potential
\begin{equation}\label{doublewellphysical}
{\mathcal W}(c) = \frac{\theta}{2} \left[ (1+c) \,\ln(1+c) + (1-c)\ln(1-c)\right] - \frac{\theta_{0}}{2}c^2  \quad  \quad ( 0 < \theta < \theta_{0}).
\end{equation}
For the connection of the quartic potential in \eqref{doublewell} to the Flory-Huggins logarithmic potential, we refer the reader to \cite{giorgini2020weak} and the references therein.
The fourth-order polynomial $W$ and the logarithmic form $\mathcal W$ are depicted in Fig.~\ref{logarithmicvspolynomial}. Here, the red curve corresponds to the polynomial potential in \eqref{doublewell} and the blue to the logarithmic potential in \eqref{doublewellphysical} for $\theta = 0.4$ and $\theta_{0} = 1.0$.
\begin{figure}[h!]
\centering
\includegraphics[width=0.4\textwidth]{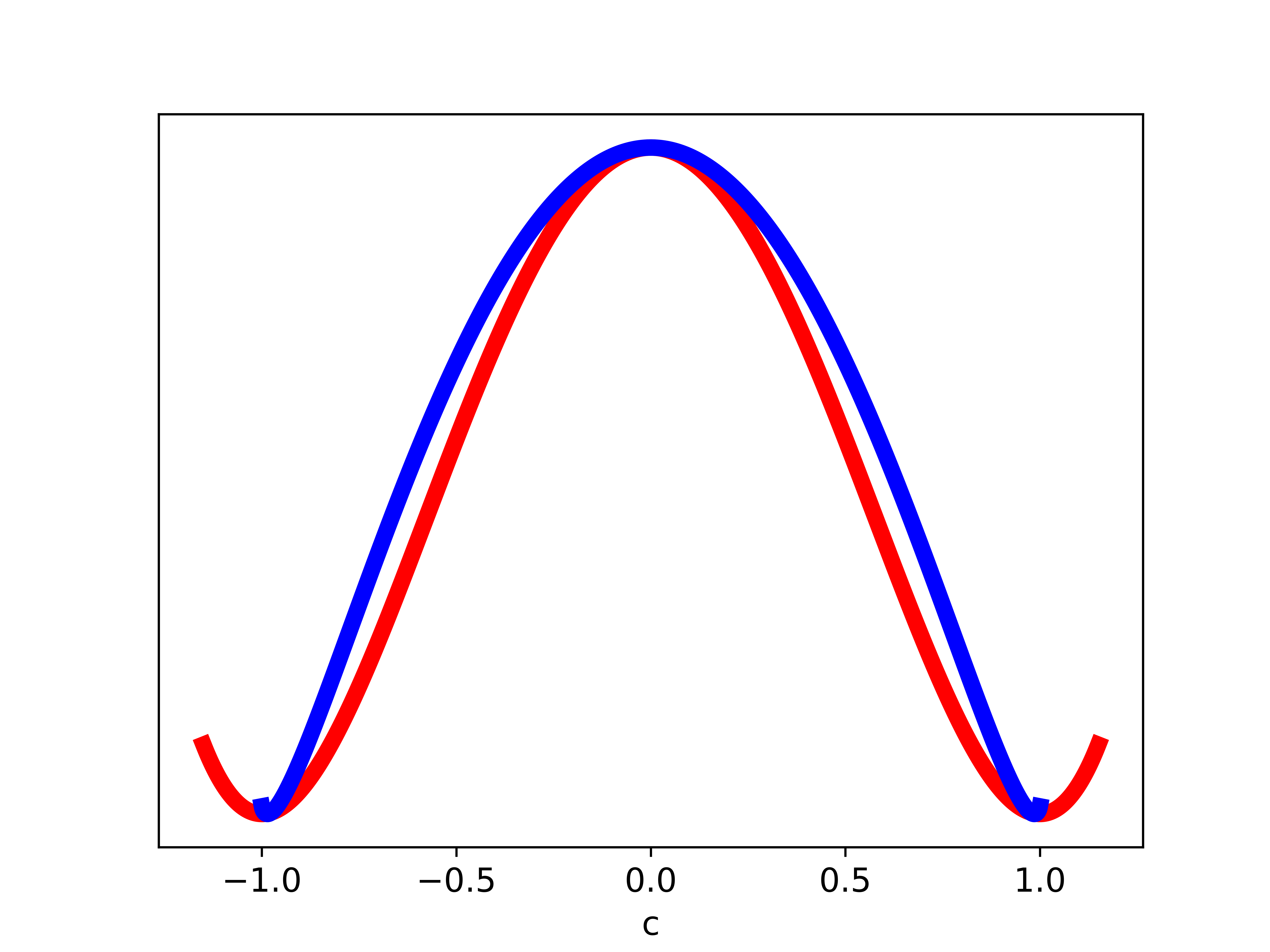}
\caption{Graphs of the logarithmic free energy $\mathcal W$ (red) and the quartic free energy $W$ (blue).}
\label{logarithmicvspolynomial}
\end{figure}

To close the  NSCH system \eqref{NSCH}, we assume 
homogeneous Dirichlet boundary condition for $\bm{u}$ and homogeneous Neumann boundary conditions for $c$ and $\mu(c)$, which are
\vskip-1em
\begin{alignat}{2}
\label{boundaryconditionsNSCH}\bm{u} = \bm{0} , \;  \grad c \cdot \bm{n} = 0, \; \grad \mu(c)  \cdot \bm{n} &= 0
\qquad &&\text{on} \; \partial \Omega \times (0,T).
\end{alignat}
Appropriate initial conditions for $\vecu$ and $c$ have to be added. 


\subsection{Energy dissipation for the NSCH system \texorpdfstring{\eqref{NSCH}}{} \label{sec_energyNSCH}}
A major step towards well-posedness of a system of evolution equations in continuum mechanics 
is the verification of an energy dissipation estimate. This can also be understood as 
thermodynamical consistency, i.e., a statement on the validity of the second law of thermodynamics.\\ 
%
For the NSCH system, we define the energy expression
\[
     E[\vecu^T,c,\vecd^T] =  \frac12 {|\vecu|^2}  + W(c) + \frac{\gamma}{2} {|\vecd|^{2}}
     \quad \left(\vecu \in \setR^d, \, c\in [-1,1], \, \vecd \in \setR^d  \right),
\]
which is just the classical van der Waals energy if the velocity vanishes and $\vecd$ is identified with the gradient of the phase field variable.
Then 
we have the following statement on thermodynamical consistency.

\begin{theorem}[Energy dissipation for \eqref{NSCH}]
\label{EDNSCH-theorem}
Let $\nu, \gamma > 0$ and let  $ {\bm U} = (p,\vecu^T,c)^T: \bar \Omega_T \to \mathcal U $ be a classical solution of the initial boundary value problem for (\ref{NSCH}) with the boundary conditions from \eqref{boundaryconditionsNSCH}.\\
Then, we have for all $t \in (0,T)$ the energy dissipation rate
\end{theorem}
\begin{equation}\label{EnergydissipationNSCH}
\dfrac{d}{dt} \int_{\Omega} E [\vecu^T,c, (\grad c)^T](\bm{x},t) \; d\bm{x} =  - \int_{\Omega} \left( |\grad \mu(c(\bm{x},t))|^2 + \nu |\grad \bm{u}(\bm{x},t)|^{2} \right) d\bm{x} \leq 0,
\end{equation}
The energy statement \eqref{EnergydissipationNSCH} is well-known. We recall here its proof to enable 
a comparison with the proofs of the Theorems \ref{EDNSEK-theorem}, \ref{EDrelaxedNSEK} below that contain corresponding statements for approximative models. 
\medskip 

\parindent0em
\textit{Proof} (of Theorem \ref{EDNSCH-theorem}).
We  derive the energy inequality  \eqref{EnergydissipationNSCH} for the NSCH system by multiplying \eqref{NSCH}$_2$ with $\bm{u}$. Integration over $\Omega$  leads to 
\begin{equation*}
\dfrac{d}{dt} \int_{\Omega} \frac12   {|\bm{u}|^2} \; d\bm{x} + \int_{\Omega} ((\bm{u} \cdot \grad) \bm{u}) \cdot \bm{u} \; d\bm{x} + \int_{\Omega} \grad p \cdot \bm{u} \; d\bm{x}   = - \int_{\Omega} c\grad \mu(c) \cdot \bm{u} \; d\bm{x} +  \nu \int_{\Omega} \Delta \bm{u} \cdot \bm{u} \; d\bm{x}. 
\end{equation*}
The use of the  theorem of Gau\ss, the divergence constraint $\div \vecu =0$ and the boundary condition 
on $\vecu$ from \eqref{boundaryconditionsNSCH} imply 
\begin{equation}\label{est1}
\begin{array}{rcl}
\ds \dfrac{d}{dt} \int_{\Omega} \frac12 {|\bm{u}|^2} \; d\bm{x}
&=& \ds - \int_{\Omega} c\grad \mu(c) \cdot \bm{u} \; d\bm{x}
   - \int_{\Omega} \nu |\grad \bm{u}(\bm{x},t)|^{2}  d\bm{x}.  
\end{array}
\end{equation}
Now, we  multiply \eqref{NSCH}$_3$ with  $\mu(c)$ and get after integration 
by the definition of $\mu$ the equation
\begin{equation}\label{est2}
\dfrac{d}{dt} \int_{\Omega}   W(c)  + \frac{\gamma}{2} |\grad c|^2     \; d\bm{x}+
\int_{\Omega} \div{(c \bm{u})} \mu(c) \; d\bm{x}    = - \int_{\Omega}  |\grad \mu(c(\bm{x},t))|^2  d\bm{x}.
\end{equation}
For the last line, we used the Neumann boundary conditions in \eqref{boundaryconditionsNSCH}.
The result \eqref{EnergydissipationNSCH} follows now from 
adding up \eqref{est1} and \eqref{est2} using once again the boundary conditions \eqref{boundaryconditionsNSCH}.
\qed
\medskip

We see that the energy $E$ decreases in time along solutions of the NSCH system. The dissipation rate is driven by the gradient in the chemical potential and the gradient in the velocity times the viscosity parameter. We consider equation \eqref{EnergydissipationNSCH} as an a-priori estimate on solutions of the NSCH system which is the ground for any well-posedness result. For the sake of completeness, we note a recent result on the existence of strong global-in-time solutions 
 proven in~\cite{giorgini2020weak}. By a strong solution we mean a triple $(p,\vecu^T,c)^T: \bar\Omega_T \to \mathcal U$  such that all equations in \eqref{NSCH} hold almost everywhere in $\Omega_T$.
To formulate the result, we introduce the Hilbert spaces $\bm{V}_{\theta}$ and $\bm{H}_{\theta}$, 
which are 
the closures of $C^{\infty}_{0, \theta} (\Omega)$ with respect to the $L^2(\Omega)$-norm
and with respect to the $H^1_{0}(\Omega)$-norm, respectively.
Thereby, $C^{\infty}_{0, \theta} (\Omega)$ is the space of divergence-free vector fields in $C^{\infty}_{0}(\Omega)$.  
The space  $C([0,T]; (W^{2,q}(\Omega))_{w})  $
consists of all functions $f\in L^\infty([0,T]; W^{2,q}(\Omega))$ such that 
$t\mapsto \langle \varphi,  f  \rangle  $ is continuous for all $\phi$ in the dual space of $W^{2,q}(\Omega)$. \\ We adopt the following  theorem  from~\cite{giorgini2020weak} for our initial boundary value problem.

\begin{theorem}[The existence of strong solutions of the NSCH system]
\label{theexistencetheorem}
Let $\Omega$ be a domain with $C^3$-boundary in $\mathbb{R}^2$.  Assume that initial data $\bm{u}_{0} \in \bm{V}_{\theta}(\Omega)$ and $c_{0} \in H^{2}(\Omega)$ are given such that $ \grad c_{0} \cdot \bm{n} = 0$ holds on $\partial \Omega$ and such that $\mu_{0} =  W'(c_{0})-\gamma \Delta c_{0} $ is in $H^{1}(\Omega)$ with $W$ from \eqref{doublewell}.\\ 
Then, there exists a strong solution $(p,\bm{u}^T,  c)$ of the initial boundary value problem for  (\ref{NSCH}), \eqref{boundaryconditionsNSCH}  satisfying
\begin{align*}
&p \in L^{2}(0,T; H^1(\Omega)),\\
&\bm{u} \in C([0,T]; \bm{V}_{\theta}) \cap L^{2}(0,T; H^2(\Omega)) \cap H^1(0,T;\bm{H}_{\theta}),\\
&c \in C([0,T]; (W^{2,q}(\Omega))_{w}) \cap H^1(0,T;H^1(\Omega)),
\end{align*}
for any $q \in [2,\infty)$.\\ 
The boundary conditions \eqref{boundaryconditionsNSCH} on the phase-field variable hold almost everywhere in  $\partial \Omega \times (0,T)$.
\end{theorem}

We conjecture that analogous well-posedness results can be derived for the approximative systems which we will introduce in the next two sections.


\section{A friction-type approximation of the   Navier-Stokes-Cahn-Hilliard system  \label{sec_NSEK}}
\renewcommand{\Upsilon}{\epsilon}
In this section, we introduce an approximative system for the NSCH system \eqref{NSCH}.
We localize the Navier--Stokes part of \eqref{NSCH} by employing the artificial compressibility 
approximation in the sense of \cite{chorin1967jcp}. The coupled fourth-order Cahn--Hilliard equation in  \eqref{NSCH} is substituted by a friction-type system. This approach exploits the variational structure  of Euler--Korteweg systems which has been established in, e.g.,~\cite{lattanzio17} or 
in homogenization approaches for two-phase flows, see e.g.~\cite{MR4182904}. 
We present the complete description of the friction-type approximation in Section 
\ref{sec_formNSEK}, and we provide an associated energy-dissipation result in Theorem \ref{sec_NSCH}
of Section \ref{NSEKapproximation}.
\subsection{Formulation of the friction-type approximation\label{sec_formNSEK}}
For the approximate system, we introduce two small parameters, namely the artificial-viscosity parameter $\alpha>0$ and the friction parameter $\delta >0$. We collect both in the parameter vector $\Upsilon = (\alpha, \delta)^T$.
Let the state space be given by \[
\tilde {\mathcal U}= \left\{ (p,\vecu^T,c,\vecv^T)^T \in \setR \times \setR^d\times \setR \times \setR^d \, |\,  c \in [-1,1]  \right\}.\] 
We search for the pressure   $p^{\Upsilon}=p^{\Upsilon}(\bm{x},t)$,  the velocity 
$\bm{u}^{\Upsilon} = \bm{u}^{\Upsilon}(\bm{x},t)$, the phase field variable  
$c^{\Upsilon} = c^{\Upsilon}(\bm{x},t)$ and  the flux variable $\bm{v}^{\Upsilon} = \bm{v}^{\Upsilon}(\bm{x},t)$ such that ${\bm{U}}^\Upsilon =  (p^\Upsilon,(\vecu^\Upsilon)^T,c^\Upsilon,(\vecv^\Upsilon)^T)^T:  \Omega_T \to  \tilde {\mathcal U} $
satisfies  the  evolution equations 
\begin{equation}\label{NSEKapproximation}
\begin{array}{c}
\begin{array}{rcccl}
  p_{t}^{\Upsilon} &+& \ds \frac{1}{\alpha} \divergence \vecu^{\Upsilon} &= &0, \\[2.0ex]
  \vecu_{t}^{\Upsilon} &+&   \ds (\vecu^{\Upsilon} \cdot \grad) \vecu^{\Upsilon}
   +\frac12(\grad\cdot \vecu^\Upsilon)\vecu^\Upsilon 
  + \grad p^{\Upsilon}  &=&  
        - c^{\Upsilon} ( \grad ( W'(c^{\Upsilon}) - \gamma \Delta c^{\Upsilon} ))+\nu \Delta \vecu^{\Upsilon},  \\[2.0ex]
        c_{t}^{\Upsilon} &+& \divergence (c^{\Upsilon} \vecu^{\Upsilon}) + \divergence  \vecv^{\Upsilon} &=& 0,  \\[2.0ex] 
 \bm{v}^{\Upsilon}_{t} &+& \ds  \frac{1}{\delta} 
 \grad (W'(c^{\Upsilon}) - \gamma \Delta c^{\Upsilon}) &=& \ds -\frac{1}{\delta} \vecv^{\Upsilon}
\end{array}\qquad \text{in} \; \Omega_{T}.
\end{array}
\end{equation}
The  free energy $W$ 
and the parameters $\nu,\gamma$ are chosen as in Section \ref{sec_NSCH}.\\
Appropriate initial data have to be prescribed. We  impose the boundary conditions 
\begin{alignat}{2}
\label{boundaryconditionsEK}  \vecu^{\Upsilon} = \bm{0}, \; \grad c^{\Upsilon} \cdot \bm{n} = 0, \,   \vecv^{\Upsilon} \cdot \vecn = \bm{0}     \text{ on } \; \partial \Omega \times (0,T).
\end{alignat}
In comparison with the NSCH system \eqref{NSCH}, the friction-type approximation \eqref{NSEKapproximation} contains 
only differential operators up to order three. However, the state space $\tilde {\mathcal U}$ is now ($2d+2$)-dimensional. 
\begin{remark} \label{rem_sec2}
We expect that solutions $\bm{U}^\Upsilon$ of the initial boundary value problem for 
\eqref{NSEKapproximation} converge to a solution $\bm{U}$ of the initial boundary value problem for 
\eqref{NSCH} as $|\Upsilon|$ tends to 0.\\
Let us first consider the passage $\alpha \to 0$. If we ignore the term $ - c^{\Upsilon} ( \grad ( W'(c^{\Upsilon}) - \gamma \Delta c^{\Upsilon} ))$ from the right hand side of $\eqref{NSEKapproximation}_2$, we get the classical 
artificial-compressibility approximation of the incompressible Navier--Stokes system.
 If we then  resolve  \eqref{NSEKapproximation}$_4$ for $\vecv^\Upsilon$ and plug the result in 
 \eqref{NSEKapproximation}$_3$ we obtain
 \[
 c_{t}^{\Upsilon} + \divergence (c^{\Upsilon} \bm{u}^{\Upsilon}) -
 \divergence \left(
 \grad ( W'(c^{\Upsilon}) - \gamma \Delta c^{\Upsilon}    )    \right) =  \delta \vecv^\Upsilon_t.
\]
If the friction parameter $\delta$ tends to zero we (formally) attain in the limit the 
convective Cahn--Hilliard equation as in  \eqref{NSCH}$_3$.
\end{remark}
\subsection{Energy dissipation for the friction-type system \texorpdfstring{\eqref{NSEKapproximation}}{}}

The friction-type system \eqref{NSEKapproximation} obeys  a natural 
energy dissipation law. To formulate it we introduce the 
energy term 
\[
     E^{\Upsilon}[p,\vecu^T,c,\vecv^T,\vecd^T] = \frac{\alpha}{2}{p^{2} } + \frac12 {|\vecu|^2} + \dfrac{\delta}{2}|\vecv|^{2} + W(c) + \frac{\gamma} {2} {|\vecd|^{2}}
     \quad \left((p,\vecu,c,\vecv,\vecd) \in \setR\times \setR^d\times [-1,1] \times \setR^d \times \setR^d)\right).
\]
Note that all terms in the energy   $E^{\Upsilon}$ are quadratic expresssions except the free energy function $W$. 

\begin{theorem}[Energy dissipation for \eqref{NSEKapproximation}]
\label{EDNSEK-theorem}
Let $\nu, \gamma > 0$. For $\Upsilon \in (0,\infty)^2$ let $\bm{U}^{\Upsilon}: \bar \Omega_{T} \to \tilde {\mathcal{U}}$ be a classical solution of the initial boundary value problem for (\ref{NSEKapproximation}) with the boundary conditions from  (\ref{boundaryconditionsEK}).\\
Then we have for all $t \in (0,T)$ the energy dissipation rate
\begin{equation}\label{EDNSEK}
\dfrac{d}{dt} \int_{\Omega} E^{\Upsilon}\left[(\bm{U}^{\Upsilon})^T, (\grad c^\Upsilon)^T \right] (\vecx,t) \; d\bm{x} = - \int_{\Omega}  \left(|\vecv^\Upsilon 
(\bm{x},t)|^{2} + \nu |\grad \vecu^\Upsilon (\bm{x},t)|^{2}\right) d\bm{x} \leq 0.
\end{equation}
\end{theorem}

The estimate \eqref{EDNSEK} renders the friction-type approximation to be thermodynamically consistent. Note that the energy dissipation rate coincides in the limit $|\Upsilon| \to 0$ with the rate
for the NSCH model in \eqref{EnergydissipationNSCH}.
\medskip

\textit{Proof} (of Theorem \ref{EDNSEK-theorem}). Define $\mu^{\Upsilon} = W'(c^{\Upsilon}) - \gamma \Delta c^{\Upsilon} $.  We multiply the equations  in \eqref{NSEKapproximation} with
the components of the  (variational) derivative   of  the energy $E^\Upsilon$, that are $\alpha p^{\Upsilon}$, $\vecu^{\Upsilon}$, $\mu^{\Upsilon}$,  $\delta \vecv^{\Upsilon}$, respectively, and obtain  after summing up in a straightforward way
\[
{\ds \dfrac{d}{dt} \left[(\bm{U}^{\Upsilon})^T, (\grad c^\Upsilon)^T \right]  +
\div{ \left(p^{\Upsilon} \vecu^{\Upsilon} \right) }
+\frac12\div{ \left( \vecu^{\Upsilon}  {|\vecu^{\Upsilon} |}^2\right) }
\ds {}+ \div{\left( c^\Upsilon \mu^\Upsilon \vecu^\Upsilon  \right)}
+ \div{\left(\mu^\Upsilon \vecv^\Upsilon \right) } 
} 
=   -   \vecv^\Upsilon \cdot\vecv^\Upsilon  +  \nu \Delta \vecu^\Upsilon \cdot \vecu^\Upsilon.
\]
From this equation we get the energy dissipation rate \eqref{EDNSEK} 
by integration over $\Omega$ and integration by parts exploiting the boundary conditions
\eqref{boundaryconditionsEK}. 
\qed
\medskip 

As noted above the system \eqref{NSEKapproximation} is still of third order and we aim --neglecting the viscosity operator-- 
for a final first-order approximation 
of the NSCH system \eqref{NSCH}. Moreover, in view of the non-monotone 
shape of $W'$, it is not likely that the first-order operator in 
the friction-type approximation results in a hyperbolic system. 
Therefore, we we come up with yet another approximation in the next section.

\section{The relaxed friction-type approximation of the   Navier-Stokes-Cahn-Hilliard system  \label{sec_relaxNSEK}}
 We proceed with the presentation of the final approximative system for the NSCH system 
 \eqref{NSCH}.  We employ a relaxation technique from \cite{rohde2010local} (see e.g. 
 \cite{DhaoudiDumbser22,KMR23}  for more recent work in this direction) to obtain a 
 first-order approximation that is constrained by a linear elliptic equation for a further
 unknown variable.  The relaxed  friction-type approximation \eqref{firstorderNSEK} will be given in Section 
 \ref{sec_formNSEKrelax}. We present then an energy-dissipation result  in Theorem \ref{EDrelaxedNSEK}
 of  Section   \ref{subsec_relaxenergydiss},   and conclude with a (partly numerical)  hyperbolicity analysis of the 
 first-order operator in Section \ref{sub_sec:hyperbolicity}.

\subsection{Formulation of the relaxed friction-type approximation\label{sec_formNSEKrelax}}

For the relaxed friction-type system we need  next to  the artificial-viscosity parameter $\alpha>0$ and the friction parameter $\delta >0$ a relaxation parameter $\beta >0$. All approximation 
parameters are summarized in the vector $\eps= (\alpha, \delta, \beta)^T$. As before we introduce a
state space  \[
\bar{\mathcal U}= \left\{ (p,\vecu^T,c,\vecv^T,\omega)^T \in \setR \times \setR^d\times \setR \times \setR^d \times \setR\, |\,  c \in [-1,1]  \right\}.\]
We search for the pressure   $ p^{\eps}=p^{\eps}(\bm{x},t)$,  the velocity 
$\vecu^{\eps} = \vecu^{\eps}(\bm{x},t)$, the phase-field variable  
$c^{\eps} = c^{\eps}(\bm{x},t)$, the flux variable $\vecv^{\eps} = \vecv^{\eps}(\bm{x},t)$ and 
the relaxation variable $\omega^{\eps}=\omega^{\eps}(\bm{x},t)$    such that ${\bm{U}}^\eps =  (p^\eps,(\vecu^\eps)^T,c^\eps,(\vecv^\eps)^T, \omega^\eps )^T:  \Omega_T \to  \bar {\mathcal U} $ solves

\begin{equation}\label{firstorderNSEK}
\begin{array}{c}
\begin{array}{rcccl}
p_{t}^{\varepsilon} &+& \dfrac{1}{\alpha} \divergence \vecu^{\varepsilon} &=& 0,\\ [2.0ex]
 \vecu_{t}^{\varepsilon} &+&  \ds  (\vecu^{\varepsilon} \cdot \grad) \vecu^{\varepsilon}
  +\frac12(\grad\cdot \vecu^\eps)\vecu^\eps 
 + \grad p^{\varepsilon} + c^{\varepsilon} \grad \left( W'\bigl(c^{\varepsilon}\bigr) + \frac1\beta  c^{\varepsilon}\right)   & = &  \ds  \frac1\beta c^\eps \grad  \omega^{\eps}  +  \nu \Delta \vecu^{\varepsilon},\\ [3.0ex]
c_{t}^{\varepsilon} &+& \divergence (c^{\varepsilon}\vecu^{\varepsilon}) + \divergence\vecv^{\varepsilon} &=& 0,  \\ [2.0ex]
\ds  \vecv^{\varepsilon}_{t} &+& \ds  \dfrac{1}{\delta} \grad \left( W'(c^{\varepsilon}) + \frac1\beta c^{\varepsilon} \right) &=& \ds -\frac{\vecv^{\varepsilon}}{\delta} +
 \frac1{\delta\beta} \grad\omega^{\varepsilon} 
 ,\\ [2.0ex]
 && \ds    -\gamma \Delta \omega^{\varepsilon}   +   \frac1\beta \omega^{\varepsilon} & = & \ds \frac1\beta c^{\varepsilon} 
\end{array}  \text{ in }  \Omega_{T}.
\end{array}
\end{equation}
The free energy $W$ and the parameters $\nu,\gamma$ are chosen as before. We complete the initial
boundary value problem  for 
the system \eqref{firstorderNSEK} by appropriate initial data and the boundary conditions 

\begin{alignat}{2}
\label{boundaryconditionsNSEKfirst}
\vecu^{\varepsilon} = \bm{0},   \; \vecv^{\varepsilon} \cdot \vecn = 0,\grad \omega^{\eps} \cdot \bm{n} &= 0&&\text{ on }  \partial \Omega \times (0,T).
\end{alignat}
Note that the Neumann conditions for the phase field variable  in \eqref{boundaryconditionsEK}   are transferred to the relaxation variable $\omega^\eps$. 
\begin{remark}
We expect that solutions $\bm{U}^\eps$ of the initial boundary value problem for 
\eqref{firstorderNSEK} converge to a solution $\bm{U}$ of the initial boundary value problem for 
\eqref{NSCH} as $|\epsilon|$ tends to 0. In Remark \ref{rem_sec2} we discussed already the partial limits $\alpha,\delta\to 0$. 
Let us therefore fix $\alpha, \delta$  and consider the relaxation limit $\beta \to 0$. We set 
$\vecu\equiv 0$   and obtain --ignoring the index $\Upsilon$ and setting $\delta =1$-- from  the decoupled equations \eqref{NSEKapproximation}$_{2}$ and \eqref{NSEKapproximation}$_{3}$
the one-dimensional model problem 
\begin{equation}\label{limit}
\begin{array}{rcl}
c_t +  v_x  &=& 0,  \\ [2.0ex]
v_{t} +  {( W'(c))}_x  &=& \ds  -  v  + {\gamma}c_{xxx}.
\end{array}
\end{equation}
The corresponding relaxed approximation writes as 
\begin{equation}\label{approx}
\begin{array}{rcl}
c^\beta_t +  v^\beta_x  &=& 0,  \\ [2.0ex]
v^\beta_{t} +  {( W'(c^\beta))}_x  &=& \ds  -  v^\beta  + \frac{1}{\beta}
{( \omega^\beta -c^\beta)}_{x},\\[2ex]
\ds  -\gamma \Delta \omega^{\varepsilon}   +   \frac {1}{\beta} \omega^{\varepsilon} & = & \ds \frac {1}{\beta} c^{\varepsilon}. 
\end{array}
\end{equation}
For $\beta \to 0$ we expect that $|\omega^\beta - c^\beta|$ tends to zero, and that 
$ \beta^{-1}|\omega^\beta - c^\beta|$ approaches $\Delta c^\beta$. These formal 
limits have been verified in e.g. \cite{ViorelRohde, Giesselmann}, where the relaxation limit for a viscous version of 
\eqref{limit} has been analyzed.  
\end{remark}

\subsection{Energy dissipation for the  relaxed friction-type  system \texorpdfstring{\eqref{firstorderNSEK}}{}} 
\label{subsec_relaxenergydiss}
As for the  models  discussed in Sections \ref{sec_NSCH}, \ref{sec_NSEK},  we can prove an energy-dissipation rate for the relaxed friction-type system \eqref{firstorderNSEK} that ensures its thermodynamical consistency. We start from the 
energy term 
\[
\begin{array}{c}
   \ds   E^{\epsilon}[p,\vecu^T,c,\vecv^T, \omega, \vece^T] =
   \ds    \frac{\alpha}{2}{p^{2} } + \frac12 {|\vecu|^2} + \dfrac{\delta}{2}|\vecv|^{2} + W(c) + 
    \frac1{2\beta} (c-\omega)^2+
     \frac{\gamma^{2}}{2} {|\vece|^{2}} \\[3ex]
 \hspace*{6cm}   \left((p,\vecu,c,\vecv,\omega,\vece) \in \setR\times \setR^d\times [-1,1] \times \setR^d \times \setR \times \setR^d)\right)
\end{array}     
\]
and provide  
%
%
%
%
%
\begin{theorem}[Energy dissipation for \eqref{firstorderNSEK}]
\label{EDrelaxedNSEK}
Let $\nu, \gamma > 0$. For $\eps \in (0,\infty)^3$ let $\bm{U}^{\eps}: \bar \Omega_{T} \to \bar {\mathcal{U}}$ be a classical solution of the initial boundary value problem for (\ref{firstorderNSEK}) with the boundary conditions from  (\ref{boundaryconditionsNSEKfirst}).\\
Then we have for all $t \in (0,T)$ the energy dissipation rate
\begin{equation}\label{Efirstorder}
\dfrac{d}{dt} \int_{\Omega} E^{\varepsilon}\left[(\bm{U}^{\varepsilon})^T, (\grad\omega^\eps)^T\right] (\bm{x},t)\; d\bm{x} 
= - \int_{\Omega}  \left(|\vecv^\Upsilon (\bm{x},t)|^{2} + \nu |\grad \vecu^\Upsilon (\bm{x},t)|^{2}\right) d\bm{x} \leq 0.
\end{equation}
\end{theorem}

\begin{proof}
As in the proofs of Theorems \ref{EDNSCH-theorem}, \ref{EDNSEK-theorem}, we multiply the 
evolution equations in \eqref{firstorderNSEK} by the  following components of the variational
derivative of $E^\eps$:
\[
\alpha p^\eps, (\vecu^\eps)^T, W'(c^\eps) + \frac1{\beta}(c^\eps-\omega^\eps), \delta (\vecv^\eps)^T. 
\]
After summing up these results and using the product rule  we get 
\begin{equation}
\begin{array}{rcl}
\ds && \ds \hspace*{-3cm}\dfrac{d}{dt} \left(  
 \frac{\alpha}{2}{(p^\eps)^{2} } + \frac12 {|\vecu^\eps|^2} + W(c^\eps)+ \dfrac{\delta}{2}|\vecv^\eps|^{2}  \right)
  + \frac{1}{\beta} c^\eps_t  (c^\eps-\omega^\eps)
\\[3ex]
&+&{} \ds \div{ \left(p^{\eps} \vecu^{\eps} \right) } 
+ \frac12\div{ \left(  {|\vecu^{\eps} |}^2  \vecu^{\eps} \right) }\\[2ex]
\ds &+& \ds   \div{ \left( c^{\varepsilon} \left( W'(c^{\varepsilon}) + \frac1{\beta} (c^{\varepsilon} 
 - \omega^\eps)\right)\vecu^\eps \right)}\\[3ex]
 &+& \ds  \divergence  \left(\left( W'\bigl(c^{\varepsilon}\bigr) + \frac1{\beta} (c^{\varepsilon} 
 - \omega^\eps)\right)  \vecv^{\varepsilon}\right)\\[3ex]
\ds 
\ds && ={}   -   \vecv^\eps \cdot\vecv^\eps +  \nu \Delta \vecu^\eps \cdot \vecu^\eps.
\end{array} \label{varia}
\end{equation}
In turn we multiply the elliptic constraint $\eqref{firstorderNSEK}_5$ by $\omega^{\eps}_{t}$ to deduce 
\[
  \frac\gamma 2\frac{d}{dt} |\grad\omega^\eps|^2    - \div{( \omega^\eps_t  \grad \omega^\eps )}+
    \frac1{\beta}  \omega^\eps_t   \left(\omega^{\varepsilon}  -c^{\varepsilon}  \right)   = 0.  
\]
The addition of the last line and \eqref{varia} directly provides  after integration and using the boundary conditions \eqref{boundaryconditionsNSEKfirst} the result \eqref{Efirstorder}. 

\end{proof}
\subsection{A hyperbolic sub-system of the  relaxed friction-type approximation \texorpdfstring{\eqref{firstorderNSEK}}{}} 
\label{sub_sec:hyperbolicity}

\newcommand{\vecQ}{{\bm Q}}
\newcommand{\vecS}{{\bm S}}
\newcommand{\vecH}{{\bm H}}
In the preceding Section \ref{subsec_relaxenergydiss}, we proved  that the relaxed friction-type system  \eqref{firstorderNSEK}  is thermodynamically consistent.  
In view of its numerical discretization another property is even more important. Let us restrict ourselves to the case $d=1$ and introduce the reduced  state space \[
{\mathcal Q} = \left\{ (p,u,c,v)^T \in \setR^4\, |\,  c \in [-1,1]  \right\}.\]
Skipping the index vector $\eps$ and using $x= x_1$ the equations $\eqref{firstorderNSEK}_1-\eqref{firstorderNSEK}_4$  can be rewritten for 
$\vecQ: \Omega_T \to \mathcal Q$ in the  form   
\begin{equation}\label{rewrittenfirstorderNSEK}
{\bm{Q}}_t + {\vecf(\vecQ)}_{x}  = \vecS(\vecQ, \omega),
\end{equation}
with
\[
\vecf(\vecQ) = \left(\begin{array}{c}
\dfrac1\alpha u\\[1.5ex]
\ds  \frac34 u^2 + p + G(c)\\[1.5ex]
cu+ v\\[1.5ex]
\ds \frac{1}{\delta} \left(W'(c) + \dfrac1\beta c \right)\\
\end{array}
\right)  ,\quad 
\vecS(\vecQ, \omega) = 
\left(\begin{array}{c}
0\\[1.5ex]
\ds \frac1\beta c  \omega_{x} + \nu u_{xx}\\[1.5ex]
0\\[1.5ex]
\ds \frac{1}{\delta} v+ \frac{1}{\delta\beta} \omega_{x}
\end{array}
\right).
\]
The function $G$ is supposed to satisfy 
\[
G'(c) = c W''(c) + \frac{1}{\beta}c.
\]
We will prove that the closed  homogeneous system of conservation laws given by 
\begin{equation} \label{claw}
{\bm{Q}}_t + {\vecf(\vecQ)}_{x}= 0
\end{equation}
is hyperbolic in $\mathcal Q$. That means that the flux Jacobian $D\vecf $, 
which is given by 
\begin{equation}\label{flux_con_prime}
    D\bm{f}(\vecQ) = \begin{pmatrix} 
        0 &  \ds \frac1\alpha  & 0 & 0 \\[1.5ex]
        1 & \ds \frac32 u & G'(c)& 0 \\
        0 & c  & u & 1 \\
        0 & 0  & \ds \frac{1}{\delta} \left( W''(c) + \dfrac1\beta \right) & 0\\ 
    \end{pmatrix},
\end{equation}
satisfies 
for all $\vecQ\in \mathcal Q$ the spectral conditions
\newcommand{\vecr}{{\vecstyle{r}}}
\[
\lambda_1(\vecQ), \ldots, \lambda_4(\vecQ) \in  \setR, \qquad {\rm{span}}
\{\vecr_1(\vecQ),\ldots,\vecr_4(\vecQ) \} = \setR^4. 
\]
Here we denote for $i \in \{1,\ldots,4\}$ by $\lambda_i=\lambda_i(\vecQ)$ and $\vecr_i = \vecr_i(\vecQ)$ the eigenvalues and corresponding eigenvectors of $D\vecf(\vecQ)$, respectively.\\
Hyperbolicity of  \eqref{claw} would ensure that one can use 
standard numerical methods for  hyperbolic conservation laws to discretize the  corresponding sub-system in \eqref{rewrittenfirstorderNSEK} in a robust manner. Note that the evaluation of the non-conservative products involving $c\omega_x $  in the right-hand-side term $\vecS$ in \eqref{rewrittenfirstorderNSEK} pose no problems because $\omega$ is a more regular function as a solution of the linear elliptic constraint $\eqref{firstorderNSEK}_5$. The following theorem ensures 
the hyperbolicity for $\beta$ chosen sufficiently small. To formulate it we need the notion 
of an entropy/entropy-flux pair. The pair $(\eta,q): {\mathcal Q} \to \setR^2$ is called entropy/entropy-flux pair for \eqref{claw}
iff  $\eta$ is convex and  the compatibility relation 
\begin{equation}\label{compa}
   (\grad \eta(\vecQ))^T D\vecf(\vecQ)    = (\grad q(\vecQ))^T 
\end{equation}
holds for all $\vecQ \in \mathcal Q$.

\begin{theorem}[Hyperbolicity of \eqref{claw}]
    \label{theorem_hyperbolicity}
    Let $\alpha,\delta >0$ and let  $\beta$ satisfy
    \begin{equation}\label{betasmall}
     \beta < -    \left(\min_{c\in [-1,1]}  \left\{  \phantom{a^2_1} \!\!  \!\!\!\!\! \min\{W''(c),0\} \right\}\right)^{-1},
    \end{equation}
    Then, the  pair $(\eta,q): {\mathcal Q} \to \setR^2$ 
    with 
    \begin{equation}
    \begin{array}{rcl}\label{etaAndq}
    \eta (\vecQ) &=&  \ds  \frac{\alpha}{2}{p^{2} } + \frac12  {u}^2  + W(c)+ \frac{1}{2\beta} c^2 + \frac{\delta}{2} {v}^2,  \\[2ex]
    q(\vecQ)   &=&  \ds p u + \frac12 u^3 + c \left(W'(c) + \dfrac1\beta c \right) u + \left(W'(c) + \dfrac1\beta c \right) v.
    \end{array}
    \end{equation}
    is an entropy/entropy-flux pair for \eqref{claw}. \\
    The system of first-order conservation laws \eqref{claw} is hyperbolic in $\mathcal Q$.
%
%
%
%
 %
\end{theorem}

\begin{proof}   
The Hessian matrix    $\bm{H} \eta $ of $\eta$ from \eqref{etaAndq} is given by 
\[
({\bm{H}} \eta)(\vecQ) = {\rm{diag}} (\alpha , 1, W''(c) + \beta^{-1}, \delta)   \in \setR^4      \qquad          (\vecQ \in {\mathcal Q}).
\]
Thus, the condition \eqref{betasmall}  ensures the (strict) convexity of $\eta$. \\
The flux Jacobian  $D\bm{f}(\vecQ)$ has been given in  \eqref{flux_con_prime}. 
We obtain 
\[
\begin{array}{c}
 (\grad \eta(\vecQ))^T D\vecf(\vecQ)  =  
 (\alpha p, u, W'(c) + \beta^{-1}c, \delta v )  D\vecf(\vecQ) 
  =  \begin{pmatrix} u \\[1.5ex] \ds p + \frac{3}{2} u^2 + c(W'(c)   + \beta^{-1} c) \\[1.5ex] u G'(c)  + (W'(c) + \beta^{-1}c  ) u + v(W''(c) + \beta^{-1}) \\[1.5ex] W'(c) + \beta^{-1}c 
    \end{pmatrix}^T.
\end{array}
\]
With the definition of $G$ we observe that the last vector in the last line is the transpose 
of $\grad q$. We have proven \eqref{compa}. 
The hyperbolicity of \eqref{claw} is now a consequence of the existence of an entropy/entropy-flux pair, see e.g. \cite{Godunov,FriedrichsLax}. 
\end{proof}


Theorem \ref{theorem_hyperbolicity} ensures the hyperbolicity of \eqref{claw} but does not give any insights on the exact form of (real) eigenvalues of the flux Jacobian $D\vecf$.

In order to find  the  eigenvalues of \eqref{rewrittenfirstorderNSEK}, we need the Jacobian of the flux, given by $ 
D\vecf(\vecQ)$ in \eqref{flux_con_prime}.  We proceed to compute 
the characteristic polynomial for $D\vecf(\vecQ)$ denoted 
by $p= p(\lambda;\vecQ)$. We get
\begin{equation}\label{firstdetermi}
    p(\lambda;\vecQ) = {\rm{det}}(D\vecf \left(\vecQ)- \mathbb{I} \lambda \right) = \begin{vmatrix}
        -\lambda & \ds \frac1\alpha & 0 & 0\\[1.5ex]
        1 & \ds \frac32 u - \lambda& G'(c) & 0 \\[1.5ex]
        0 & c & u - \lambda & 1\\[1.5ex]
        0 & 0 & \ds \frac1\delta \left(W''(c)+\frac1\beta \right) & -\lambda 
    \end{vmatrix}.
\end{equation}
Eliminating  the first column and the first two rows  to develop  the determinant we get 
from \eqref{firstdetermi} the relation
\begin{equation}\label{seconddetermi}
   p(\lambda;\vecQ) = -\lambda \begin{vmatrix}
      \ds   \frac32 u -\lambda & G'(c) & 0 \\[1.5ex]
        c & u - \lambda & 1\\[1ex]
        0 & \ds \frac1\delta \left(W''(c)+\frac1\beta \right) & -\lambda 
    \end{vmatrix} - \begin{vmatrix}
      \ds   \frac1\alpha & 0 & 0\\[1.5ex]
         c & u - \lambda & 1\\[1.5ex]
          0 & \ds  \frac1\delta \left(W''(c)+\frac1\beta \right) & -\lambda 
    \end{vmatrix}.
\end{equation}
Using  the co-factors of $\frac32 u- \lambda$ and $c$ in the first determinant  and the co-factor of $\alpha^{-1}$
in the second determinant in \eqref{seconddetermi} we arrive at

\[
     p(\lambda;\vecQ) = -\lambda \left(
     \left(\frac32 u- \lambda\right) \begin{vmatrix}
     u - \lambda & 1\\[1.5ex]
   \ds   \frac1\delta \left(W''(c)+\frac1\beta \right) & -\lambda
     \end{vmatrix} - c \begin{vmatrix}
         G'(c) & 0 \\[1.5ex]
    \ds      \frac1\delta \left(W''(c)+\frac1\beta \right) & -\lambda
     \end{vmatrix}
     \right)
     - \dfrac{1}{\alpha}\begin{vmatrix}
         u - \lambda & 1\\[1.5ex]
    \ds      \frac1\delta \left(W''(c)+\frac1\beta \right) & -\lambda
     \end{vmatrix}.
\]
We  calculate the determinant of each  $(2\times2)$-matrix and 
are led to 
\[
    p(\lambda;\vecQ) = \left( \lambda^2 - \frac32 u \lambda - \dfrac{1}{\alpha} \right) \left(\lambda^2 - u\lambda - \left(\dfrac{1}{\delta}(W''(c) + \frac{1}{\beta}   \right) \right) - c \; G'(c) \lambda^2.
\]
After some algebraic calculations and substituting the definition of $G'$ from  \eqref{rewrittenfirstorderNSEK}, we finally obtain 
$p(\cdot;\vecQ)$ as 

\begin{equation}\label{charactpolynomial}
 \begin{array}{rcl}
    p(\lambda;\vecQ) &=& \ds \lambda^4 -  \frac52 u \lambda^3  + 
    \left(\frac32u^2  -\left(W''(c) + \frac1{\beta} \right) \left(\frac1{\delta} + c^2\right) - \frac1{\alpha}\right) \lambda^2  \\[2.5ex] 
   && \ds \hspace*{4cm} +
    \left(\dfrac{3u}{2\delta} \left((W''(c)+ \frac{1}{\beta} \right)   + \dfrac{u}{\alpha}\right) \lambda + \dfrac{1}{\alpha \delta}\left(W''(c) + \frac1\beta \right).
\end{array}
\end{equation}
Although  $p(\cdot;\vecQ)$  depends on $c$ and $u$ only (but not on the "linear-flux"-unknowns $p,\, v$), we have not been able to determine all four roots of $p(\cdot,\vecQ)$ for arbitrary $\vecQ \in \mathcal Q$, but only for very specific states.  
Instead  of focusing on these states  we present some numerical 
experiments for  the spectrum of $D\vecf(\vecQ)$.


\begin{Example}[Eigenvalue computations] \label{example}
In view of the absence of an explicit eigenvalue formulae, we alternatively present numerical calculations of  the roots of $p(\cdot; \vecQ)$ for varying values of the parameters $\alpha,\delta,\beta$ and 
  the variables $u$ and $c$. Throughout the example, we fix $\beta = 0.01$ satisfying \eqref{betasmall}.  \\
 To visualize the behavior of the roots 
 we show in  all figures a color plot of the sign 
of the characteristic polynomial over $\lambda$: 
The yellow (magenta) color corresponds to positive (negative) values of $p(\cdot; \vecQ)$. 
\medskip

%
%
%


\textbf{Static velocity:}
First, in two numerical experiments we analyze the impact of $\alpha$ and $\delta$ on the roots for  the static  choice $u=0$ and $c=1$.
%
%
%
%
%
 %
\begin{figure}[h!]
    \centering
    \begin{minipage}[b]{0.48\textwidth}
        \includegraphics[width=0.8\textwidth]{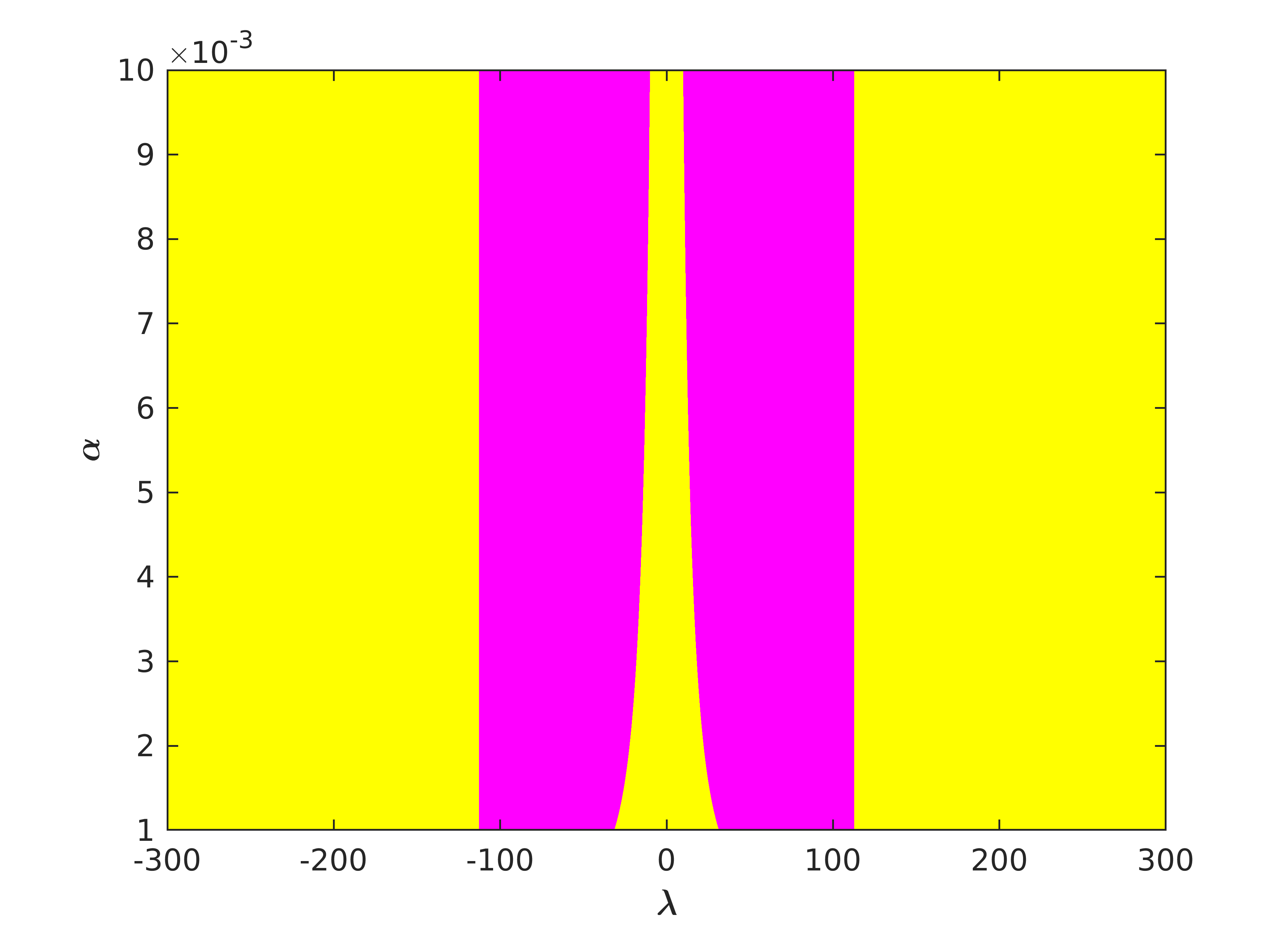}
        \centering
        \caption{Color plot of ${\rm{sgn}}(p(.; \vecQ))$ for fixed $\delta = 0.09 $ and  $\alpha \in [0.01,0.001]$ on the  vertical axis.}
        \label{fig:eigen.alpha_steady}
    \end{minipage} \quad
        \begin{minipage}[b]{0.48\textwidth}
            \includegraphics[width=0.8\textwidth]{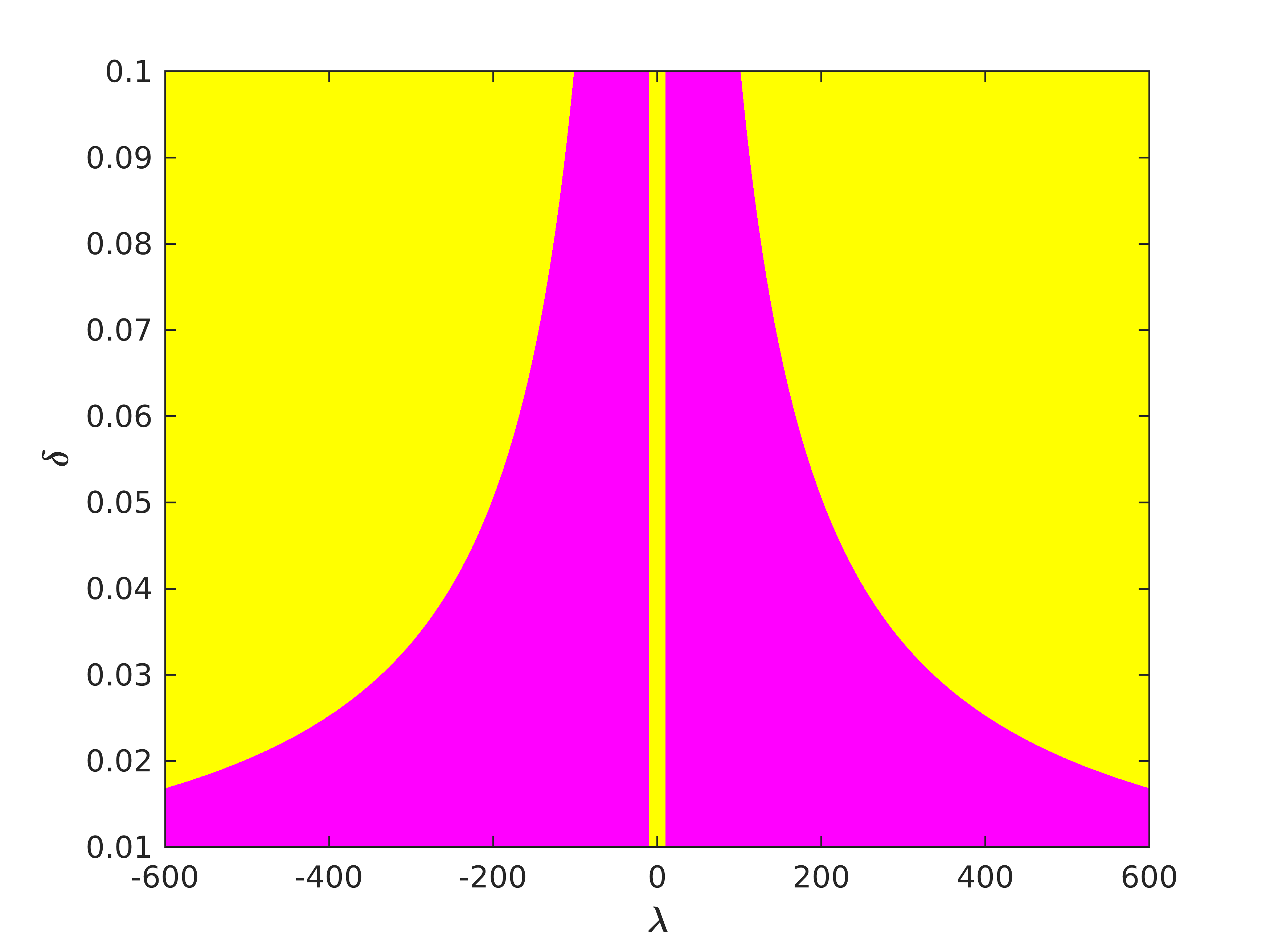}
            \centering
            \caption{Color plot of ${\rm{sgn}}(p(.; \vecQ))$ for fixed $\alpha = 0.01 $ and  $\delta\in [0.01,0.001]$ on the  vertical axis.}
            \label{fig:eigen.delta_steady}
        \end{minipage}
\end{figure}
Figure~\ref{fig:eigen.alpha_steady} illustrates 
for fixed $\delta$, that $p(\cdot; \vecQ)$  has 
four real roots for any choice of $\alpha$. 
Interestingly, the bigger roots in absolute values seem not to depend on $\alpha$ whereas the  
inner pair of roots explodes with $\alpha \to 0$
and might, eventually, cross the outer ones. 
%
%
Note also that the roots are symmetric with respect to the horizontal  $(\lambda=0)$-axis, see the effect of setting $u=0$ in \eqref{charactpolynomial}. \\ 
In Figure~\ref{fig:eigen.delta_steady} we display the results 
for fixed $\alpha$, varying now $\delta$. Again  $p(\cdot; \vecQ)$  has 
four  (symmetric) real roots for any choice of $\delta$. Now the outer roots depend sensitively on $\delta$, exploding for $\delta \to 0$.\\
The observations on the asymptotic behaviour confirm   just that solutions of the limit problem \eqref{NSCH} 
exhibit the inifinite-speed-of-propagation property. The major challenge for the numerical simulation based on the relaxed friction-type system \eqref{firstorderNSEK} will be therefore the treatment of the fast wave-speeds.
\medskip 

\textbf{Non-zero velocity:}
%
%
%
%
Now we keep all data but  change the velocity,  first to $u=-15$.

\begin{figure}[h!]
    \centering
    \begin{minipage}[b]{0.48\textwidth}
        \includegraphics[width=0.8\textwidth]{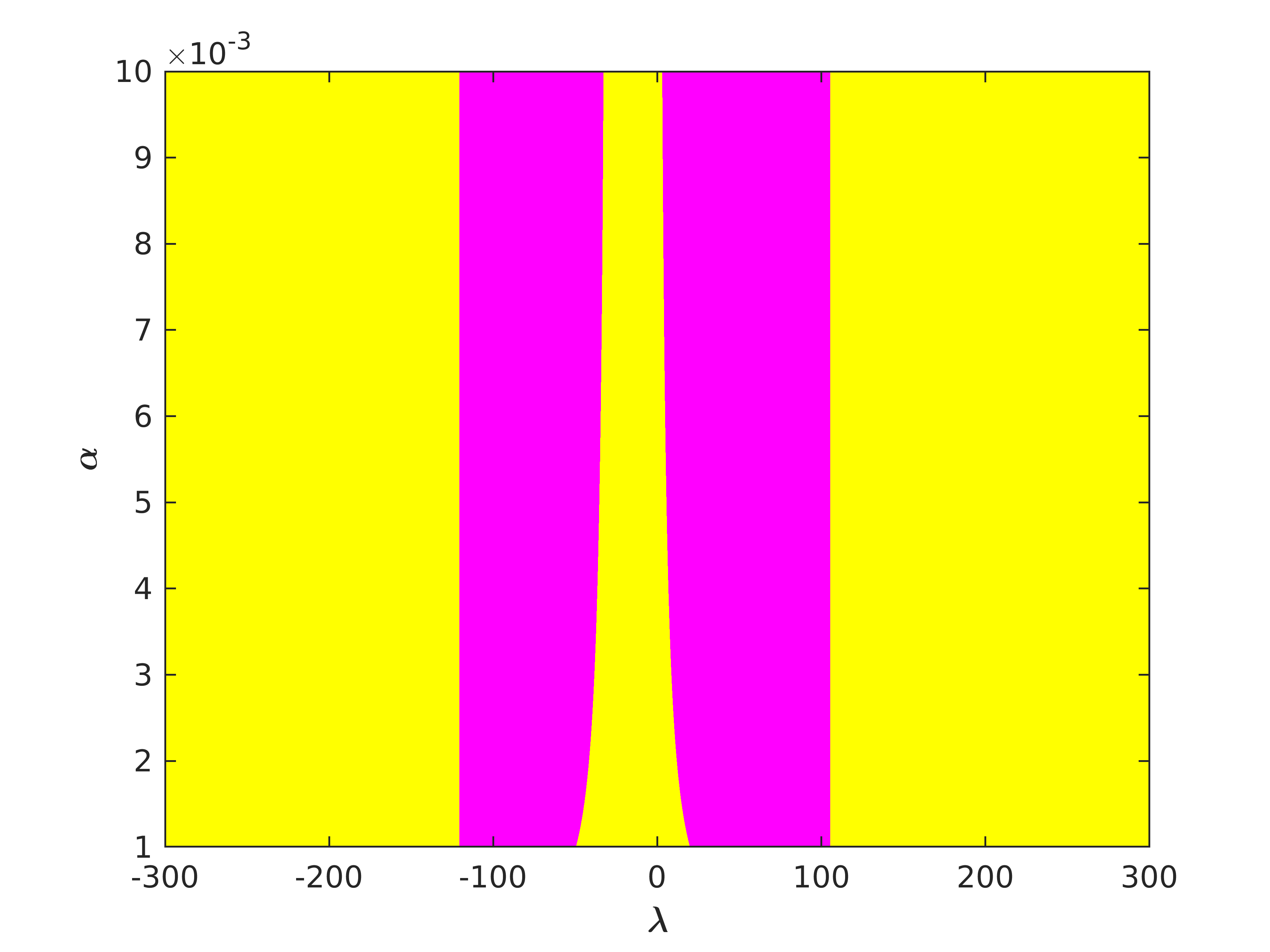}
        \centering
        \caption{Color plot of ${\rm{sgn}}(p(.; \vecQ))$ for fixed $\alpha = 0.01 $ and  $\delta\in [0.01,0.001]$ on the  vertical axis.}
        \label{fig:eigen.alphachanges}
    \end{minipage} \quad
        \begin{minipage}[b]{0.48\textwidth}
            \includegraphics[width=0.8\textwidth]{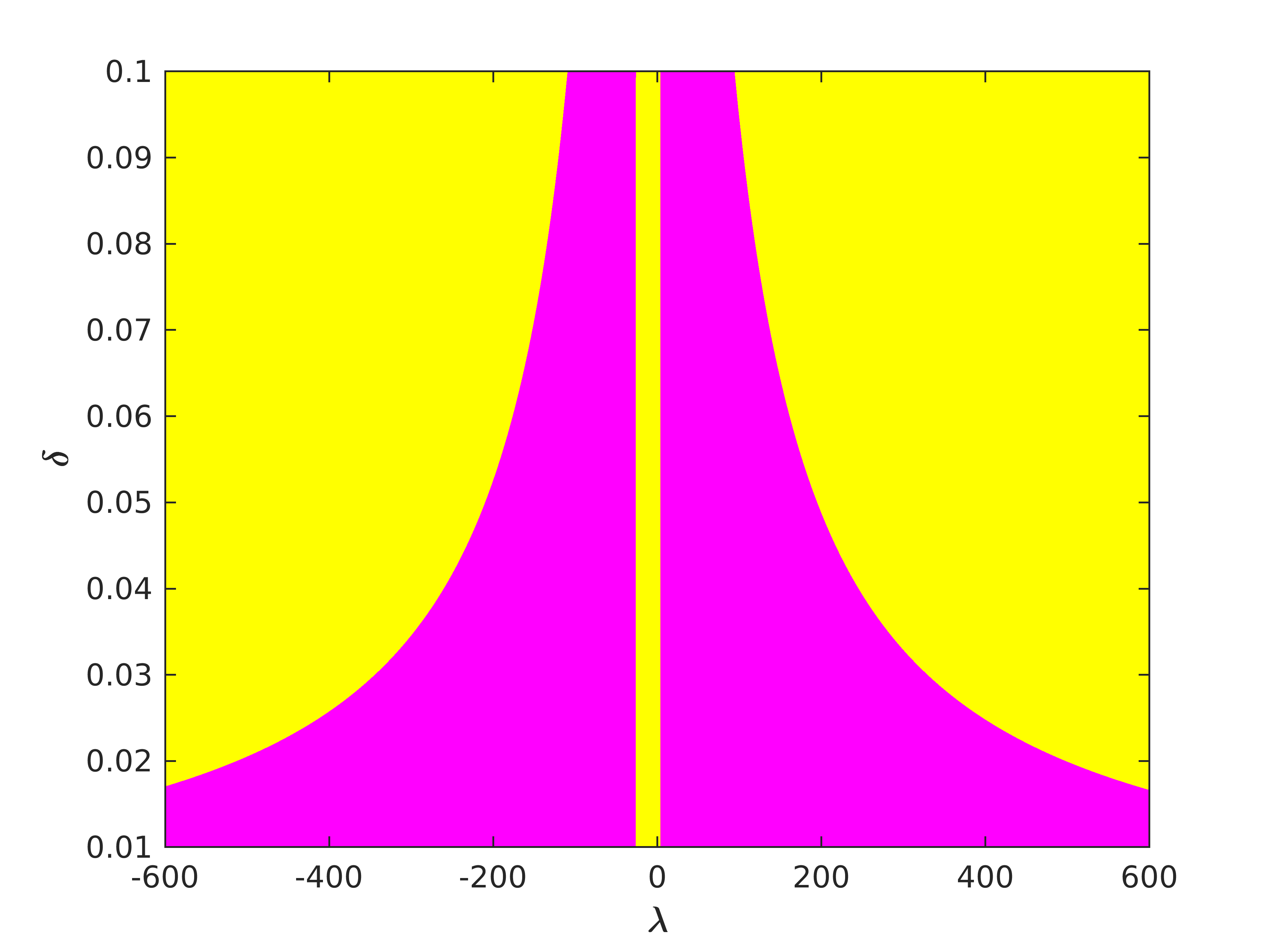}
            \centering
            \caption{Color plot of ${\rm{sgn}}(p(.; \vecQ))$ for fixed $\alpha = 0.01 $ and  $\delta\in [0.01,0.001]$ on the  vertical axis.}
            \label{fig:eigen.deltachanges}
        \end{minipage}
\end{figure}
Like in the previous setting, we detect four real and distinct roots in Figure~\ref{fig:eigen.alphachanges} for any $\alpha$ having the same asymptotic behaviour. However, the roots are not symmetric anymore, but shifted to the left, see e.g. the small absolute-value pair.
In Figure \ref{fig:eigen.deltachanges} we observe the analogous behaviour when varying $\delta$. \\
To confirm the expected shift effect for positive velocity $u=12$ in positive direction we  display the corresponding results in Figures \ref{fig:eigen.example3}
and \ref{fig:eigen.example4}.

\begin{figure}[h!]
   \centering
   \begin{minipage}[b]{0.48\textwidth}
    \includegraphics[width=0.8\linewidth]{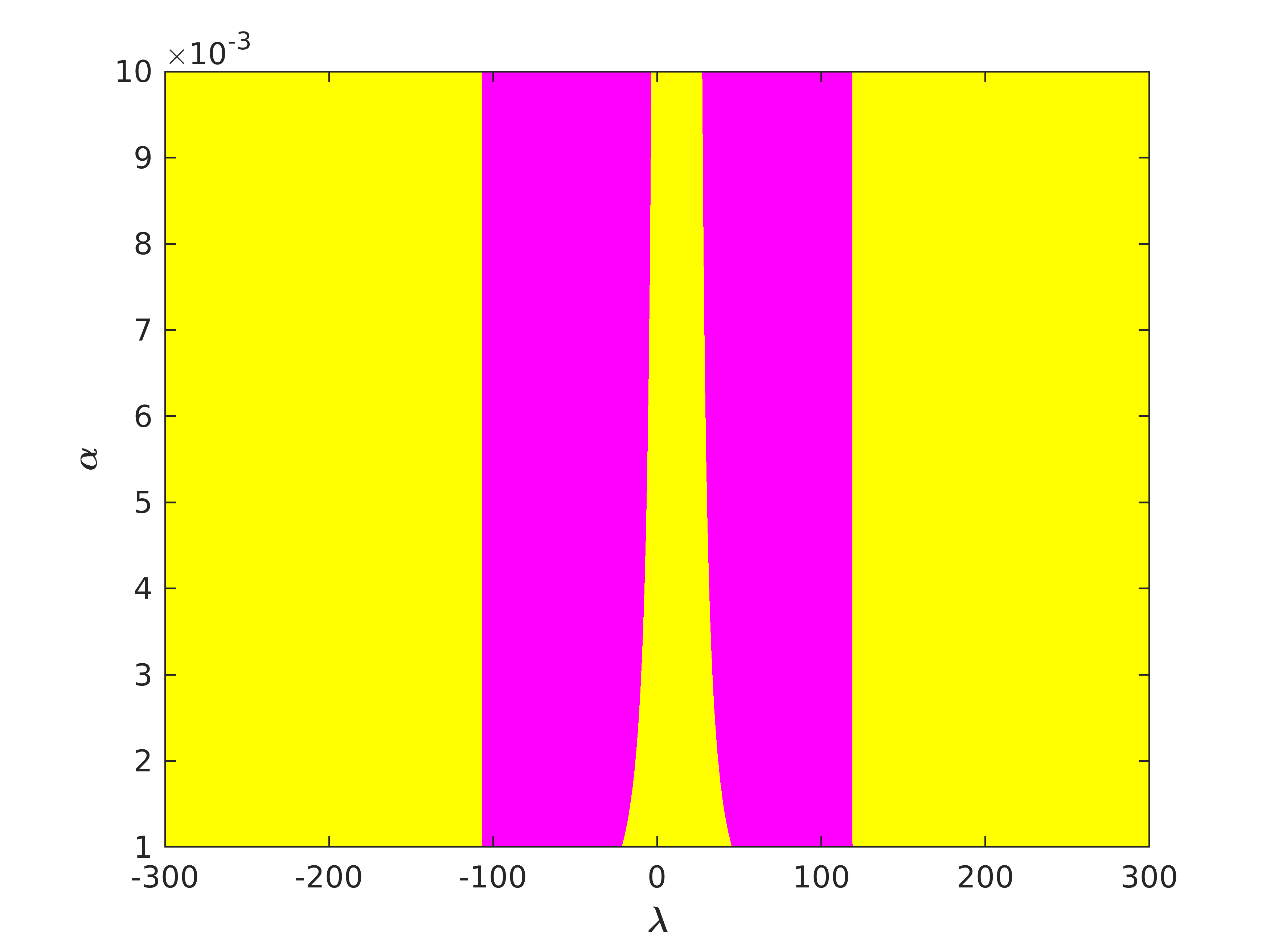}
    \centering
    \caption{Color plot of ${\rm{sgn}}(p(.; \vecQ))$ for fixed $\delta = 0.01 $ and  $\alpha\in [0.01,0.001]$ on the  vertical axis.}
    \label{fig:eigen.example3}
    \end{minipage}
    \begin{minipage}[b]{0.48\textwidth}
        \includegraphics[width=0.8\textwidth]{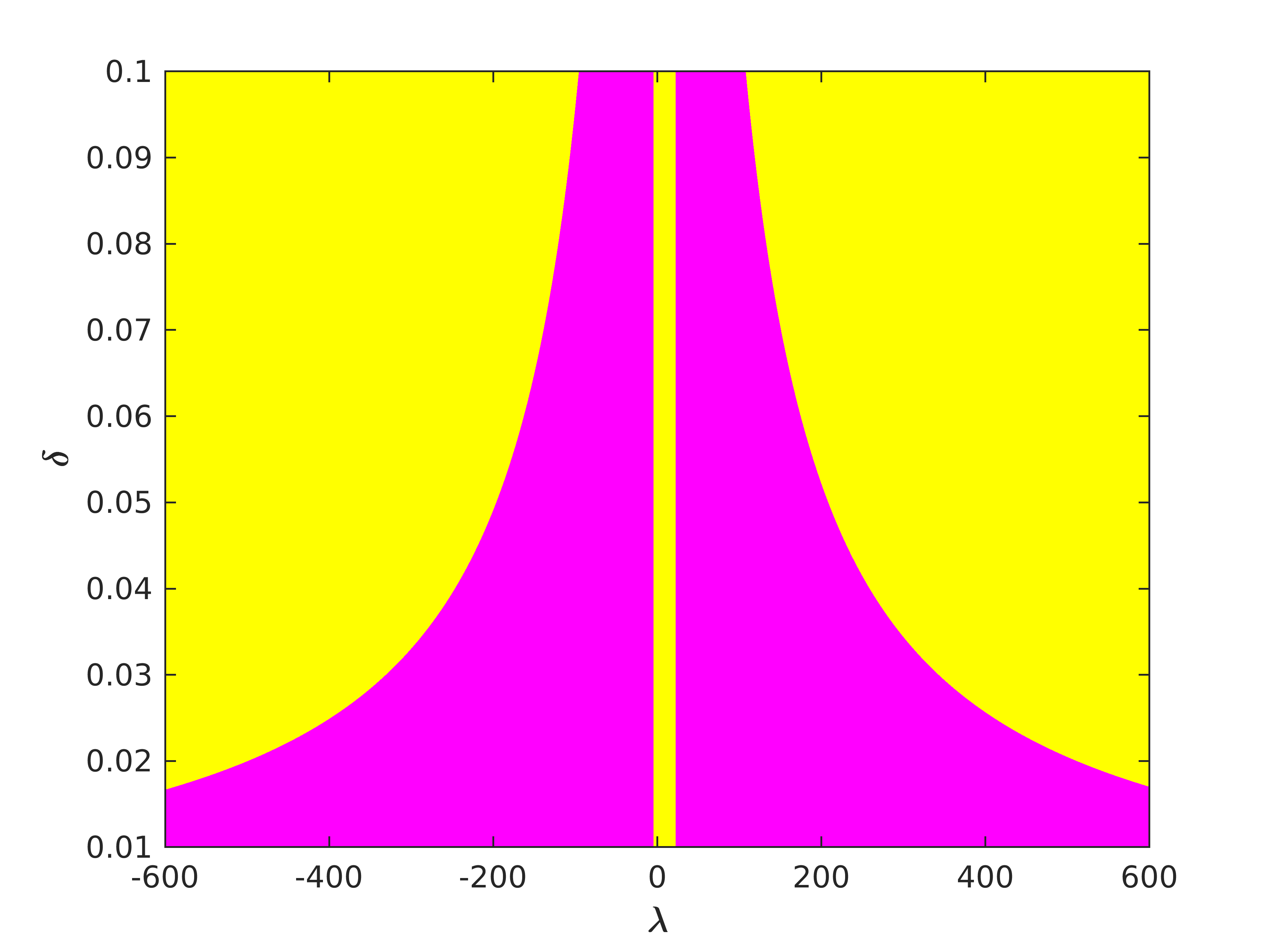}
        \centering
        \caption{Color plot of ${\rm{sgn}}(p(.; \vecQ))$ for fixed $\alpha = 0.01 $ and  $\delta\in [0.01,0.001]$ on the  vertical axis.}
        \label{fig:eigen.example4}
    \end{minipage}
\end{figure}

%
%
%
%
The numerical experiments  up to now indicate that the roots  of $p(\cdot; \vecQ)$ are arranged in two pairs: one  pair depends on $\alpha$ but not on $\delta$; the other pair shows reverse behaviour. We believe that this property holds in the entire state space. 
\medskip 

%
%


{\textbf{Effect of $c$-variation:}}
Finally we stress that the qualitiative behaviour of the root locations is not effected 
by the choice of $c$ which is not surprising in view of the dominating effect of $\beta^{-1}$.
In Figures \ref{fig:eigen.example5} and  \ref{fig:eigen.example6} we display the results 
of the same parameter choices as before, but with $c=-0.5$.
%
\begin{figure}[h!]
   \centering
   \begin{minipage}[b]{0.48\textwidth}
    \includegraphics[width=0.8\linewidth]{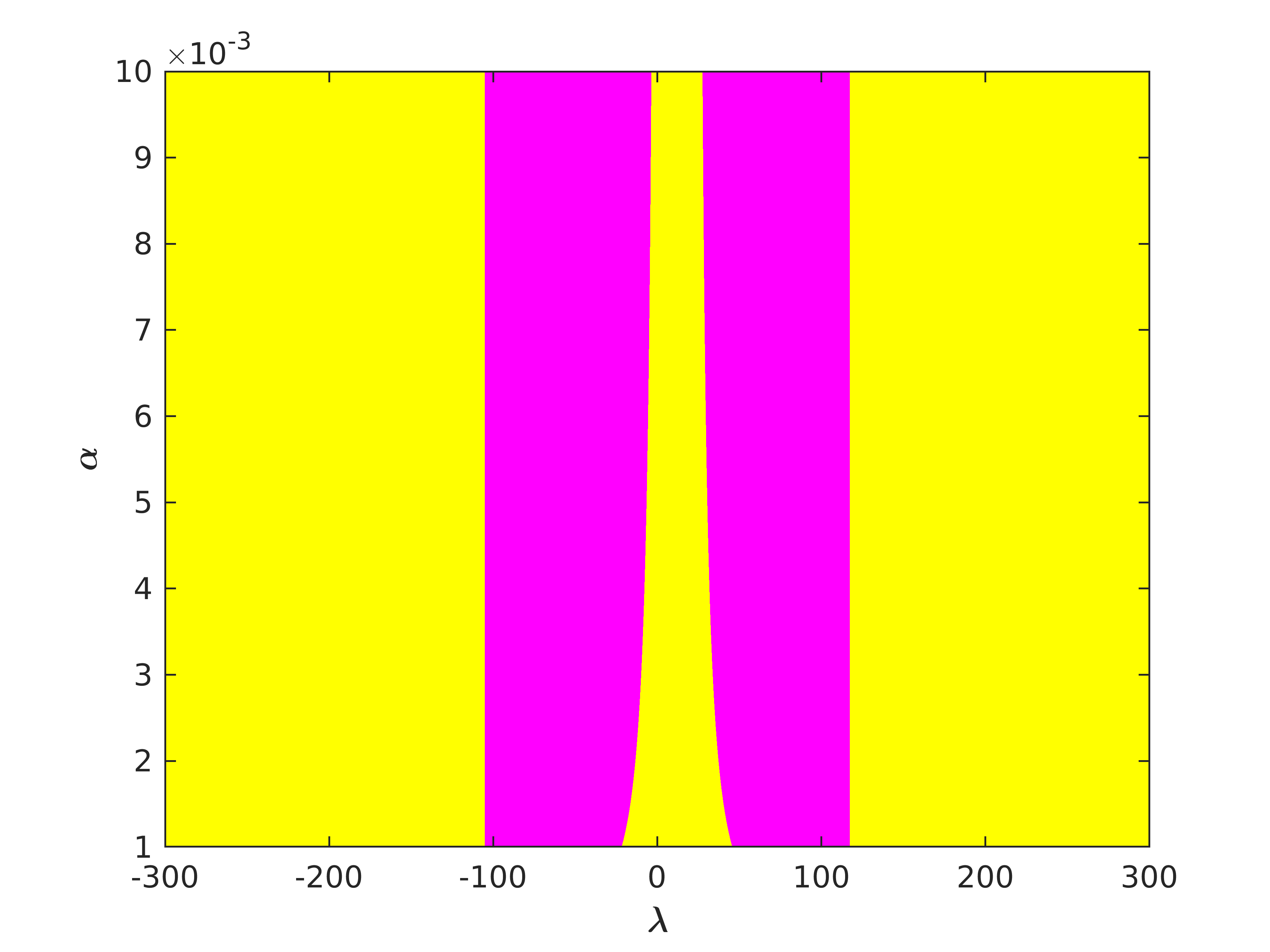}
    \centering
    \caption{Color plot of ${\rm{sgn}}(p(.; \vecQ))$ for fixed $\delta = 0.01 $ and  $\alpha\in [0.01,0.01]$ on the  vertical axis.}
    \label{fig:eigen.example5}
    \end{minipage}
    \begin{minipage}[b]{0.48\textwidth}
        \includegraphics[width=0.8\textwidth]{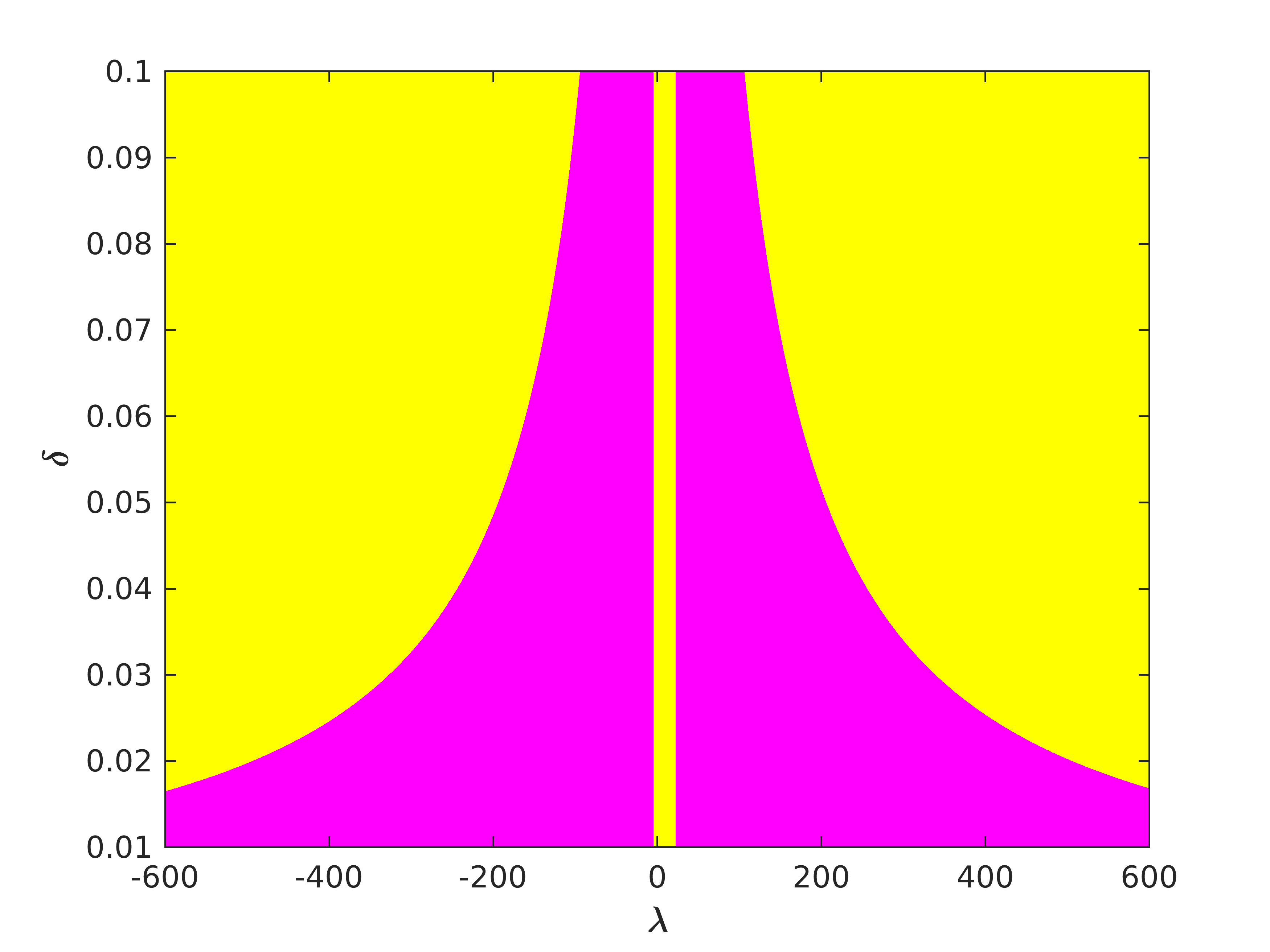}
        \centering
        \caption{Color plot of ${\rm{sgn}}(p(.; \vecQ))$ for fixed $\alpha = 0.01 $ and  $\delta\in [0.01,0.01]$ on the  vertical axis.}
        \label{fig:eigen.example6}
    \end{minipage}
\end{figure}
%
\end{Example}

\section{Numerical Experiments}
\label{sec_numer.experiments}
In this section, we present numerical experiments for the first-order relaxed friction-type approximation  \eqref{firstorderNSEK}  of the NSCH equations in one spatial dimension
(see \eqref{rewrittenfirstorderNSEK}).
These experiments are on the one hand designed to prove numerically that the presented first-order relaxed friction-type model is able to accurately approximate the solution of original Navier--Stokes--Cahn--Hilliard system in the limit of $|\varepsilon| \rightarrow 0$ and on the other hand to illustrate the wave structure of its hyperbolic sub-system.
To this end, besides the full set of equations of the first-order friction-type model in \eqref{rewrittenfirstorderNSEK}, we additionally investigate two reduced versions of it.
For the first sub-system, we only consider the hyperbolic part of \eqref{rewrittenfirstorderNSEK}.
Hence, we omit any contributions due to the source terms or the elliptic constraint for $\omega^\varepsilon$.
The second system is also based on the hyperbolic sub-system however with the additional contribution of the friction-type source term $-\frac{v}{\delta}$ taken into account.
In the following, these three systems are denoted as the \textit{relaxed friction-type model}, the \textit{hyperbolic sub-system without sources} and the \textit{hyperbolic sub-system with friction}, respectively.

The considered numerical experiments are organized as follows, first, we discuss results for the \textit{hyperbolic sub-system with and without sources} with Riemann initial data and different choices of the relaxation parameters $\varepsilon$ in Section \ref{sub_sec:no_source} and \ref{sub_sec:with_source}.
This is followed by different numerical results obtained with the \textit{relaxed friction-type model} for a stationary droplet, a spinodal decomposition and a setup which features the phenomena of Ostwald ripening in Section \ref{sub_sec:NSEK}.
Finally, to illustrate the robustness of the \textit{relaxed friction-type model}, we revisit the setup with the Riemann initial data and perform simulations with the full system.
\subsection{The relaxed friction-type model without source terms}\label{sub_sec:no_source}
The present section is devoted to the numerical investigation of the hyperbolic sub-system of \eqref{rewrittenfirstorderNSEK}.
For this, we considered two different types of Riemann initial data, first, a \textit{compression-like} initial condition, given by
\begin{equation}
    \label{eq:ShockTubeOne}
   (p^{\varepsilon},u^{\varepsilon},c^{\varepsilon},v^{\varepsilon})(x,0) = \begin{cases}
         \vspace{0.25cm}
         (0,+1, 1, 0) : ~ x \leq 0,\\
         (0,-1, 1, 0) : ~ x > 0,
    \end{cases}
\end{equation}
and second, a \textit{expansion-like} initialization, defined as
\begin{equation}
    \label{eq:ShockTubeTwo}
   (p^{\varepsilon},u^{\varepsilon},c^{\varepsilon},v^{\varepsilon}) (x,0) = \begin{cases}
         \vspace{0.25cm}
         (0,-1, 2, 0), \; x \leq 0,\\
         (0,+1, 2, 0), \; x > 0.
    \end{cases}
\end{equation}
The naming of the test cases as \textit{compression} and \textit{expansion-like} is chosen in analogy to the compressible Euler equations and only due to convenience.
To ensure the hyperbolicity of the investigated sub-system, instead of the double-well potential defined in \eqref{doublewell}, we utilize a convex free-energy function, given as $W (c) = {c}^{-1}$.
For both initial conditions, different choices of relaxation parameters are investigated, namely $\alpha \in \lbrace 0.5 , 0.1 , 0.05 \rbrace$ and $\delta = 0.0625$.
The simulations are performed with a classical second-order finite volume MUSCL-Hancock scheme on the space-time domain $\Omega_T =  (-1,1) \times (0,0.15)$.
For the slope limiting the \textit{minmod} limiter is utilized and the numerical fluxes at the cell edges are approximated by the use of a Rusanov flux.
The spatial domain is discretized with $N = 10^4$ cells and the time step size is computed based on the CFL-condition with $CFL = 0.9$.
At the boundaries non-reflecting boundary conditions are employed.

The results of the \textit{compression-like} test case in \eqref{eq:ShockTubeOne}, with different choices of the relaxation parameters are depicted in Figure \ref{fig:ShockTubeOne}.
\begin{figure}
    \centering
    \includegraphics[width=0.9\linewidth]{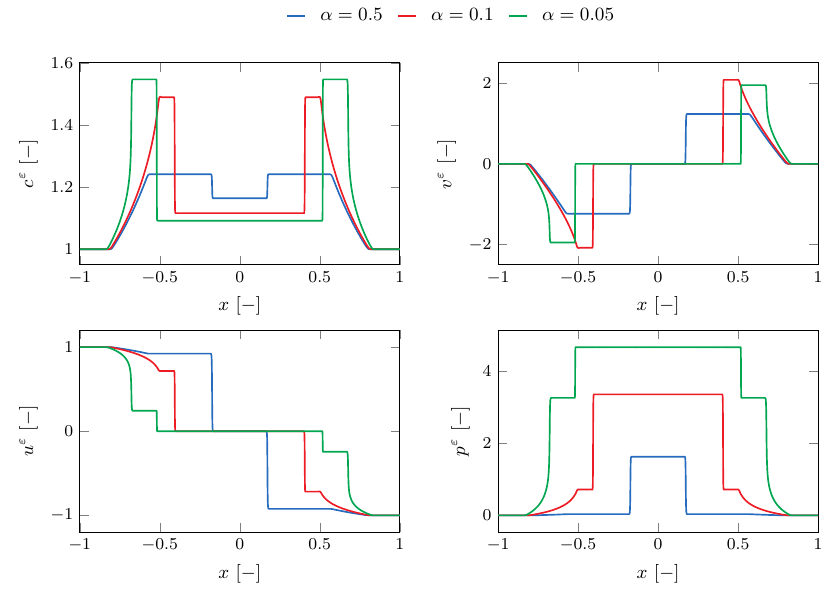}
    \caption{Comparison of the pressure $p$, the velocity $u$, the flux $v$ and the phase-field variable $c$ for the \textit{compression-like} test case and different choices of the relaxation parameters, $\alpha \in \lbrace 0.5 , 0.1 , 0.05 \rbrace$ and $\delta = 0.0625$ at $t = 0.15$.}
    \label{fig:ShockTubeOne}
\end{figure}
As expected, overall four waves can be observed in the simulation results, corresponding to the existence of  four real eigenvalues.
In the case of the initial Riemann data, defined in \eqref{eq:ShockTubeOne}, this results in two inner shock waves and two outer rarefaction waves.
The position of the inner shocks is strongly dependent on the relaxation parameter $\alpha$, while the speed of the outer rarefaction waves is almost unaffected by the choice of $\alpha$.
However, the plateau values of all conserved quantities are heavily affected by the choice of $\alpha$.
A completely different observation can be made for the \textit{compression-like} test case in \eqref{eq:ShockTubeTwo}, depicted in Fig.~\ref{fig:ShockTubeTwo}.
Here, depending on the choice of $\alpha$, two inner rarefaction waves and two outer shocks can be observed or, in the case of $\alpha = 0.05$, four rarefaction waves.
Moreover, for this setup, a strong dependence of all four wave speeds on the relaxation parameter $\alpha$ can be observed.
\begin{figure}
    \centering
    \includegraphics[width=0.9\linewidth]{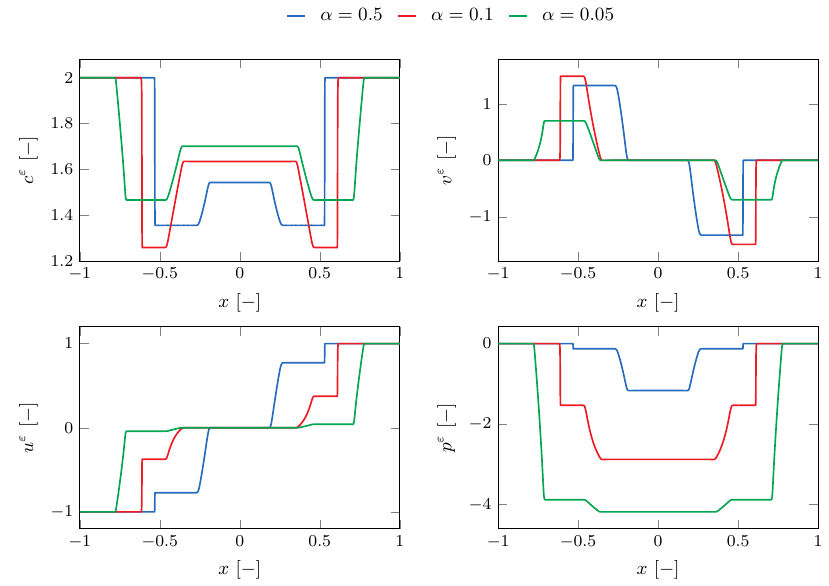}
    \caption{Comparison of the pressure $p$, the velocity $u$, the flux $v$ and the phase-field variable $c$ for the \textit{expansion-like} test case and different choices of the relaxation parameters, $\alpha \in \lbrace 0.5 , 0.1 , 0.05 \rbrace$ and $\delta = 0.0625$ at $t = 0.15$.}
    \label{fig:ShockTubeTwo}
\end{figure}
%
%
%
%
\subsection{The relaxed friction-type model with friction source term}\label{sub_sec:with_source}
While the previous section was devoted to the conservative hyperbolic sub-system of \eqref{rewrittenfirstorderNSEK}, in the following, the previous investigation is extended by taking into account the additional contribution of the friction-type source term in the balance equation of the artificial velocity.
Similarly to the previous section, the Riemann initial data in \eqref{eq:ShockTubeOne} and a spatial and temporal discretization based on a MUSCL-Hancock scheme are utilized.
However, in contrast to the previous investigations, the relaxation parameter $\alpha$ is fixed at $\alpha = 0.05$ and the relaxation parameter $\delta$ varies within $\delta \in \lbrace 0.0625, 0.01, 10^{-4}, 10^{-6} \rbrace$.
Moreover, the domain size is extended to $\Omega \in [-2.5,2.5]$.

The results for the the different choices of relaxation parameters are depicted in Figure \ref{fig:ShockTubeOneDamping}.
\begin{figure}
    \centering
    \includegraphics[width=0.9\linewidth]{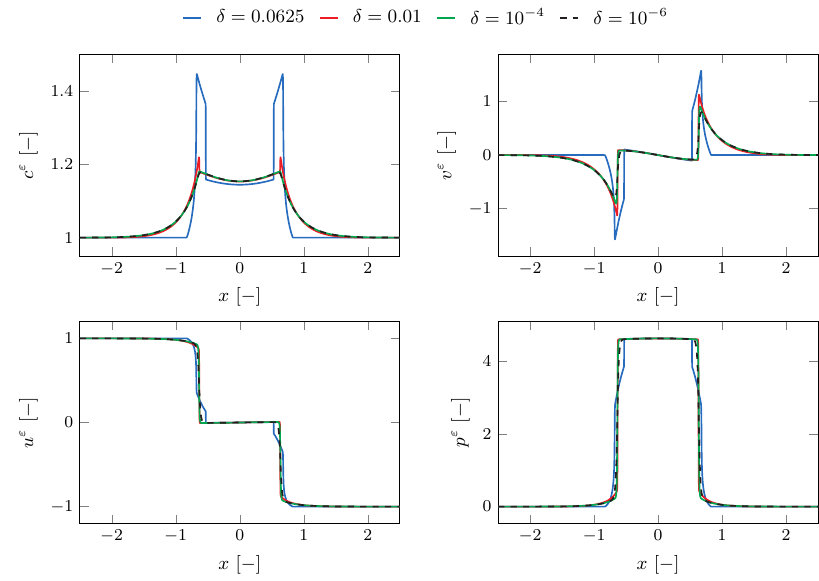}
    \caption{Comparison of the pressure $p$, the velocity $u$, the flux $v$ and the phase-field variable $c$ for \textit{compression-like} test case and different choices of the relaxation parameters, $\alpha = 0.05$ and $\delta \in \lbrace 0.0625, 0.01, 10^{-4}, 10^{-6} \rbrace$ at $t = 0.15$.}
    \label{fig:ShockTubeOneDamping}
\end{figure}
Incorporating the friction source term significantly affects the overall structure of the solution.
Moreover, the solution rapidly converges in all conserved variables to a final profile with decreasing values of $\delta$.
\subsection{The relaxed friction-type approximation of the NSCH system}\label{sub_sec:NSEK}
Finally, we consider numerical experiments with the full \textit{relaxed friction-type model} for the NSCH equations.
The goal of these test cases is on the one hand to illustrate numerically the convergence of the approximate model towards the one-dimensional NSCH equations and on the other hand to discuss the sensitivity of the results on the choice of the approximation parameters $\varepsilon$.
For this, three different types of one-dimensional test cases are considered, namely the relaxation of a droplet to equilibrium, a spinodal decomposition and a one-dimensional Ostwald ripening.
Finally, to highlight the robustness of the \textit{relaxed friction-type model}, we revisit the previously used Riemann data and show results obtained with the approximate formulation of the NSCH equations.

If not stated otherwise, in the following, the \textit{relaxed friction-type model} is discretized with a conservative second-order finite-differences scheme.
The time stepping is fully implicit by utilizing an implicit Runge-Kutta method of the Radau IIA family of fifth-order.
The domain is discretized with $N = 1000$ points and periodic boundary conditions are employed if not stated otherwise.
In addition to the presented friction-type approximation of the NSCH equations, also the original NSCH equations are considered as a reference.
However, since only one-dimensional setups with a uniform velocity field are investigated, the original NSCH equations reduce to the Cahn-Hilliard (CH) equations under these circumstances.
Hence, instead of the NSCH equations, the CH equation are solved as a reference.
For this, we utilized a conservative fourth-order finite-differences discretization as proposed in \cite{dhaouadi2024}.
The same time stepping strategy as for the friction-type approximation model is utilized.
Concerning the physical parameters, the double-well potential in \eqref{doublewell}
is considered and the capillary parameter is fixed to $\gamma = 0.001$ over all simulations.
\subsubsection{Stationary Droplet}
The first test case which we consider is a liquid droplet in its vapor which relaxes to its equilibrium state.
This test case is chosen to investigate the influence of the relaxation parameter $\beta$ on the obtained results.
For this, $\beta$ is varied in between $\beta \in \lbrace 0.1, 0.01 , 0.001 \rbrace$ while the remaining relaxation parameters, $\alpha = 10^{-4}$ and $\delta = 10^{-8}$, remained fixed.
The simulations are performed on the space-time interval $\Omega_T = (-1,1) \times (0,1)$.
The initialization of the droplet in terms of the phase-field variable is given by
\begin{align}
    c (x,0) = c^\varepsilon (x,0) = - \tanh{\left( 10 (|x|-0.5) \right)}.
\end{align}
This corresponds to a droplet with radius $r = 0.5$ whose center is initially located at $x_{ini}=0$.
The pre-factor of $10$ in the $\tanh$-function is used to diffuse the phase interface of the droplet in comparison to its equilibrium profile for the given free energy potential, in \eqref{doublewell}, and the chosen capillary parameter $\gamma$.
Moreover, for the the \textit{relaxed friction-type model}, the artificial pressure and the velocity are both initialized as $p^\varepsilon(x,0) = u^\varepsilon(x,0) =  0$.
Finally, the flux variable at $t = 0$ is defined as $v^\varepsilon (x,0) = \left( W'(c) + \frac1\beta \left(c^\varepsilon - \omega^\varepsilon \right)\right)_x$, where for the sake of clarity, the dependence of $c^\varepsilon$ and $\omega^\varepsilon$ on space and time is omitted.
Since for the evaluation of $v(x,0)$ the additional order parameter $\omega^\varepsilon$ is required,
the elliptic constraint in \eqref{rewrittenfirstorderNSEK} has to be solved for $\omega^\varepsilon$ in advance.
For this, a second-order central finite-differences discretization is utilized.
Due to the linear nature of the elliptic constraint, the corresponding coefficient matrix can be inverted once at the beginning of the simulation and reused afterwards during each time step.

The results of the equilibrated droplet at $t = 1$ are depicted in Figure \ref{fig:StationaryDroplet}.
\begin{figure}
    \centering
    \includegraphics[width=0.9\linewidth]{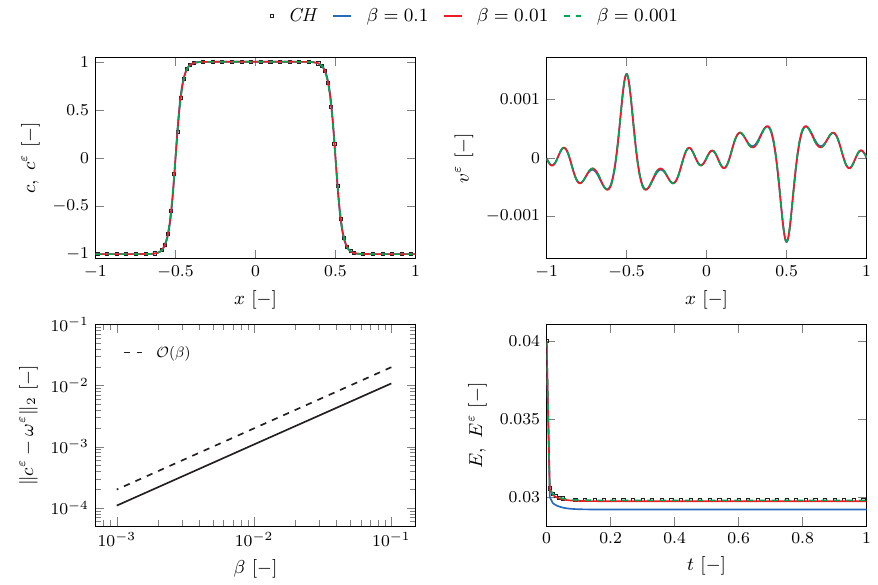}
    \caption{Comparison of the phase-field variable $c^\varepsilon$, the artificial velocity $v^\varepsilon$, the integral $L_2$ difference between the phase-field variable and the order parameter and the evolution of the total energy per unit volume $E^\varepsilon$ over time, for $\alpha = 10^{-4}$ and $ \delta = 10^{-8}$, $\gamma = 0.001$ and $\beta \in \lbrace 0.1 , 0.01 , 0.001 \rbrace$.}
    \label{fig:StationaryDroplet}
\end{figure}
Based on the visualization of the phase-field variable $c^\varepsilon$ no significant dependence of the simulation results on the relaxation parameter $\beta$ can be observed.
Moreover, all results obtained with the \textit{relaxed friction-type model} coincide well with the reference solution of the original CH equations indicated by the square symbols.
The near zero flux variable $v^\varepsilon$, which is according to \eqref{NSEKapproximation} an approximation of the gradient of the chemical potential, indicates that the droplet reached its equilibrium state.
This is also supported by the evolution of the total energy per unit volume $E^\varepsilon$, depicted in the bottom right of Figure~\ref{fig:StationaryDroplet}.
Furthermore, a monotonic decrease in the total energy can be observed.
This underlines the thermodynamical consistency of the model.
In this plot, a slight difference in the evolution of the energy with different choices of the relaxation parameter $\beta$ can be reported.
Moreover, with decreasing values of $\beta$, convergence of the solution obtained with the approximate formulation to the CH equations is reached.
Finally, in the bottom left, the expected first-order convergence of the order-parameter $\omega$ to the phase-field variable $c^\varepsilon$ with decreasing values of $\beta$ is shown.
\subsubsection{Spinodal Decomposition}
As a second test case, we consider  a one-dimensional spinodal decomposition triggered by an initially disturbance in the phase-field variable.
This setup is adopted from \cite{dhaouadi2024} and given as
\begin{align}
    c (x,0) = c^\varepsilon (x,0) =
    \begin{cases}
    +0.01\left(\sin{\left( 10 \pi (1+x) \right)} - \sin{\left( 10 \pi (1+x)^2 \right)}\right) :~x \leq 0, \\
    -0.01\left(\sin{\left( 10 \pi (1-x) \right)} - \sin{\left( 10 \pi (1-x)^2 \right)}\right) :~x > 0.
    \end{cases}
\end{align}
The remaining solution variables are initialized similar to the stationary droplet test case.
Since this numerical experiment features a strong variation of the phase-field variable over time, we use it to investigate the dependence of the simulation results on the relaxation parameter $\delta$.
For this, different choices are investigated, namely $\delta \in \lbrace 10^{-2} , 10^{-4} , 10^{-6} \rbrace$.
To avoid any disturbance due to the remaining relaxation parameters, those are kept fixed with $\alpha = \beta = 10^{-4}$.
The simulations are performed on the space-time domain $\Omega_T = (-1,1) \times (0,4)$.

The results obtained for the phase-field variable $c^\varepsilon$ at the time instances $t \in \lbrace 0 , 0.02 , 0.1 , 0.95 , 0.96 , 4 \rbrace$ are depicted in Figure \ref{fig:SpinodalDecompositionC}.
\begin{figure}
    \centering
    \includegraphics[width=0.9\linewidth]{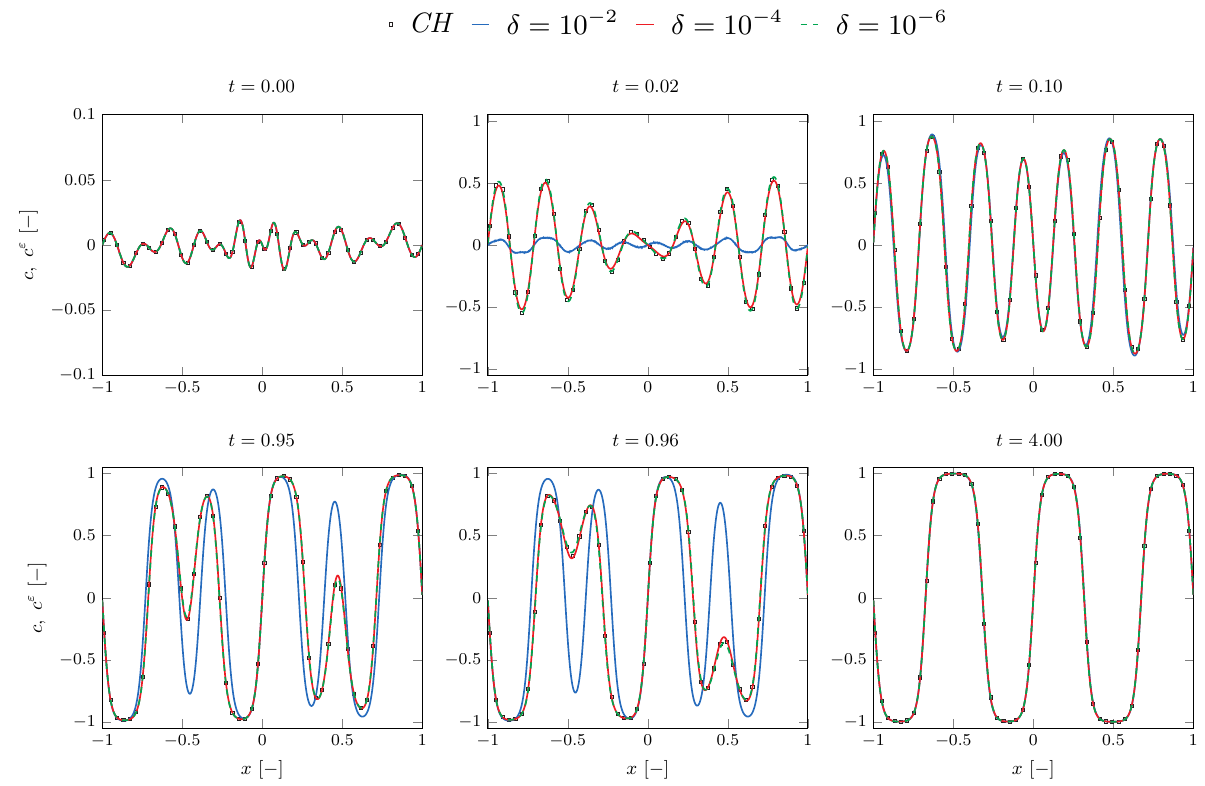}
    \caption{Comparison of the phase-field variable $c^\varepsilon$ at different time instances for $\alpha = \beta = 10^{-4}$, $\gamma = 0.001$ and $\delta \in \lbrace 10^{-2} , 10^{-4} , 10^{-6} \rbrace$.}
    \label{fig:SpinodalDecompositionC}
\end{figure}
A fast convergence of the \textit{relaxed friction-type model} to the original model can be observed for decreasing values of $\delta$.
However, distinct differences in the phase-field variable $c^\varepsilon$ can be observed at $t = 0.95$ and $t = 0.96$ for the choice of $\delta = 10^{-2}$ or $\delta = 10^{-4}$ and lower.
This is due to the fast merging of four droplets into two droplets.
In the end, overall three droplets remain independent of the choice of the relaxation parameter $\delta$.
\subsubsection{Ostwald Ripening}
As a final validation case, we consider the phenomenon of the so-called Ostwald ripening, where two distinct bubbles, each in non-equilibrium, are initialized.
Due to the non-local character of the underlying CH equation, the smaller of the two bubbles will eventually vanish and the larger one will grow until an equilibrium state is reached although they are not in contact.
We use this numerical experiment to investigate the impact of the relaxation parameter $\alpha$ on the simulation results.
Similar to the previous investigations we fix  the other relaxations parameters to $\beta = 10^{-4}$ and $\delta = 10^{-8}$ and varied $\alpha \in \lbrace 0.1 , 0.01 , 0.001 \rbrace$.
The setup is adopted from \cite{hitz2020parabolic,KMR23,dhaouadi2024}, where the initial distribution of the phase-field variable is given as
\begin{align}
    c (x,0) = c^\varepsilon (x,0) = -1 +\sum_{i=1}^2 \tanh{\left( \frac{| x - x_i |-r_i}{\sqrt{2 \gamma}} \right)},
\end{align}
and where the initial droplet positions and radii are given as $x_1 = 0.3$, $x_2 = 0.75$ and $r_1 = 0.12$, $r_2 = 0.06$, respectively.
The remaining solution variables of the \textit{relaxed friction-type model} are initialized similar to the previous investigations.
All simulations are performed on the space-time interval $\Omega_T = (0,1) \times (0,0.3)$.
The domain is discretized with $500$ grid points.

The simulation results for the phase-field variable $c^\varepsilon$ at the time instances $t \in \lbrace 0.1 , 0.2 ,0.3 \rbrace$ with the different choices for the relaxation parameter $\alpha$ are depicted in Figure \ref{fig:OstwaldRipeningC}.
\begin{figure}
    \centering
    \includegraphics[width=0.9\linewidth]{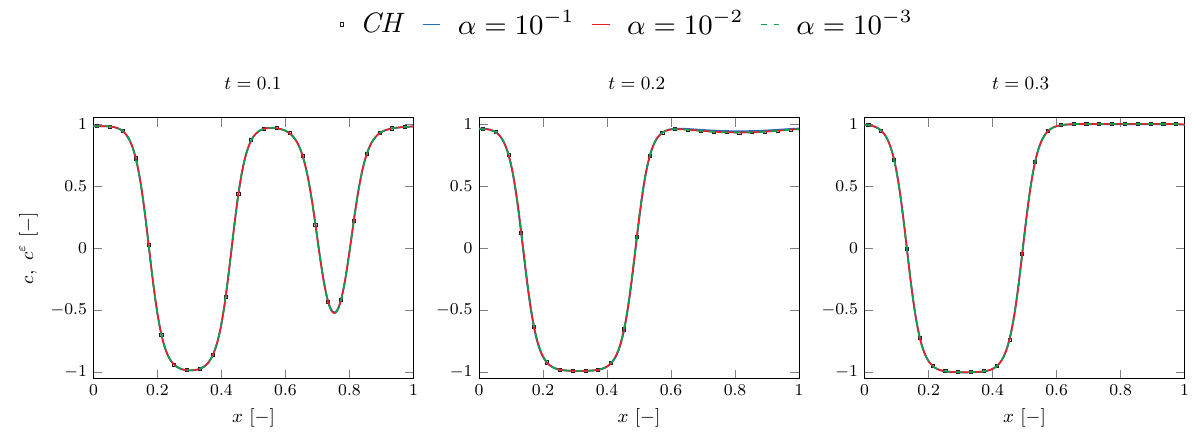}
    \caption{Comparison of the phase-field variable $c^\varepsilon$ at different time instances for $\alpha \in \lbrace 10^{-1}, 10^{-2}, 10^{-3} \rbrace$, $\beta  = 10^{-4}$, $\gamma = 0.001$ and $\delta = 10^{-8}$.}
    \label{fig:OstwaldRipeningC}
\end{figure}
Good agreement with the reference data of the CH equations is obtained for all choices of the relaxation parameter and time instances.
A slight difference of the results obtained with $\alpha = 10^{-1}$ to the remaining simulations at time $t = 0.2$ can be observed.
The velocity and the pressure at time $t = 0.2$ as well as the evolution of the energy over time are depicted in Figure \ref{fig:OstwaldRipeningE}.
Here, a consistent convergence of the velocity to a constant value of zero with decreasing values of $\alpha$ can be reported.
A similar behavior can be observed in the pressure which counteracts any disturbances in the velocity field.
\begin{figure}
    \centering
    \includegraphics[width=0.9\linewidth]{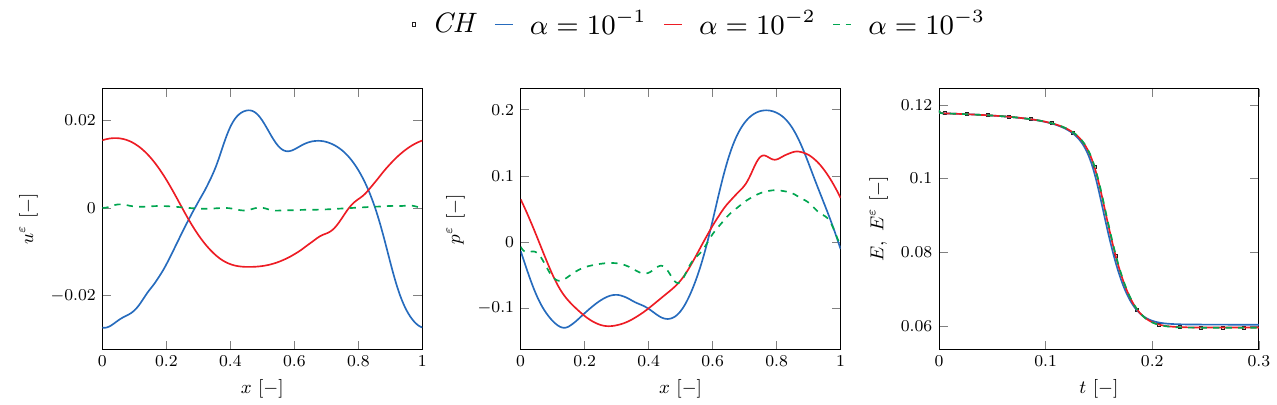}
    \caption{Comparison of the velocity $u^\varepsilon$, the artificial velocity $v^\varepsilon$ at $t = 0.2$ and the evolution of the total energy per unit volume over time for $\alpha \in \lbrace 10^{-1}, 10^{-2}, 10^{-3} \rbrace$, $\beta = 10^{-4}$, $\gamma = 0.001$ and $\delta = 10^{-8}$.}
    \label{fig:OstwaldRipeningE}
\end{figure}
Finally, similar to the previous investigations, a monotonic decrease of the total energy per unit volume can be observed, which again underlines thermodynamical consistency.

To exploit the fact, that the considered \textit{relaxed friction-type model} is also able to consider non-zero velocity distributions, we modify the previous Ostwald ripening test case in the sense that we superimpose a constant velocity of $u^\varepsilon = (0.3)^{-1}$ which causes an additional advection of the two bubbles over time.
Due to the specific choice of the velocity, the final profile of the phase-field variable at $t = 0.3$ should coincide with the one of the zero velocity setup.
The simulation results for the phase-field variable $c^\varepsilon$ at the time instances $t \in \lbrace 0.1 , 0.2 ,0.3 \rbrace$ with the different choices for the relaxation parameter $\alpha$ are depicted in Figure \ref{fig:OstwaldRipeningM}.
\begin{figure}
    \centering
    \includegraphics[width=0.9\linewidth]{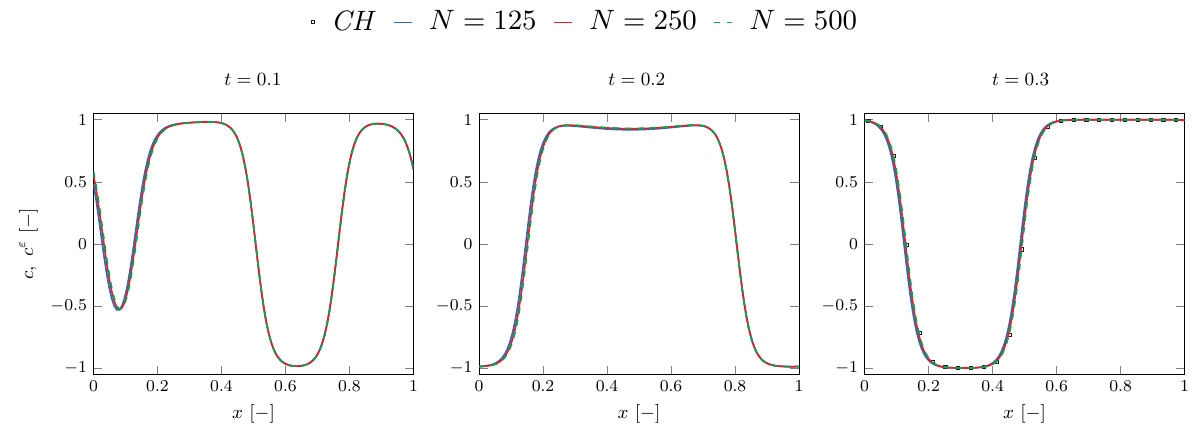}
    \caption{Evolution of the phase field $c^\varepsilon$ over time for $\alpha \in \lbrace 10^{-1}, 10^{-2}, 10^{-3} \rbrace$, $\beta = 10^{-4}$, $\gamma = 0.001$ and $\delta = 10^{-8}$ 
    and the modified Ostwald ripening setup.}
    \label{fig:OstwaldRipeningM}
\end{figure}
The advection of the two and finally single bubble is clearly visible.
At the final time $t = 0.3$, the advected bubble coincides well with the overall position of the non-moving simulation.
However, a minor mismatch in the position of the bubble is present.
This can be attributed to the inherent dispersion error of any numerical scheme, especially low-order schemes.
This observation is supported by additional simulations performed with even coarser grid resolutions.
\subsubsection{A Two-phase Shock Tube Problem}
Finally, to demonstrate the robustness of the \textit{relaxed friction-type model} and the discretization, we revisit the \textit{compression-} and \textit{expansion-like} test cases in \eqref{eq:ShockTubeOne} and \eqref{eq:ShockTubeTwo}.
For this, we use a modified double-well potential, given as
\begin{align}
    W ( c ) = \left ( c - 1 \right)^2 \left ( c - 2 \right)^2,
\end{align}
which guarantees stable equilibrium states at $c^\varepsilon = 1$ and $c^\varepsilon = 2$.
We again used a MUSCL-Hancock scheme for the discretization, however, now for the complete \textit{relaxed friciton-type model}.
The domain is discretized with $2000$ cells and for the time step restriction $CFL = 0.9$ is chosen.
The relaxation parameters and the capillary coefficient are defined as $\alpha = 0.5$, $\beta = 10$, $\gamma = 10^{-4}$ and $\delta = 0.0625$, respectively.
All simulations are performed on the space-time domain $\Omega_T = (-1,1) \times (0 , 0.15)$.
We want to stress that this initialization is not physical and only surfs the purpose to highlight the robustness of the model and the discretization.
Moreover, we would like to highlight that simulations with coarser mesh resolutions would also allow for stable simulations.
However, this would come with the price of a higher numerical diffusion and hence, more diffused shock and rarefaction wave profiles.

The results of both test cases are depicted in Figures \ref{fig:CompressionNSCH} and \ref{fig:ExpansionNSCH}, respectively.
In the case of the \textit{compression-like} test case, the nucleation of a droplet in the center of the domain, clearly visible in the phase-field variable, can be observed.
A different result is obtained with the \textit{expansion-like} setup, where the nucleation of a bubble is triggered.
For the velocity and the pressure of both test cases, a qualitatively similar however, mirrored behavior can be reported.
\begin{figure}
    \centering
    \includegraphics[width=0.9\linewidth]{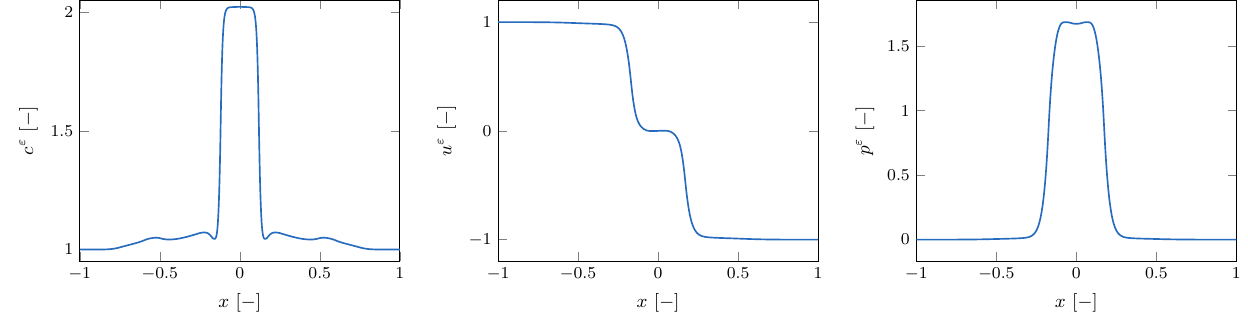}
    \caption{Evolution of the phase field $c^\varepsilon$, the velocity $u^\varepsilon$ and the pressure $p^\varepsilon$ at $t = 0.15$ for $\alpha = 0.5$, $\beta = 0.1$, $\gamma = 10^{-4}$ and $\delta = 0.0625$ for the \textit{compression-like} setup.}
    \label{fig:CompressionNSCH}
\end{figure}

\begin{figure}
    \centering
    \includegraphics[width=0.9\linewidth]{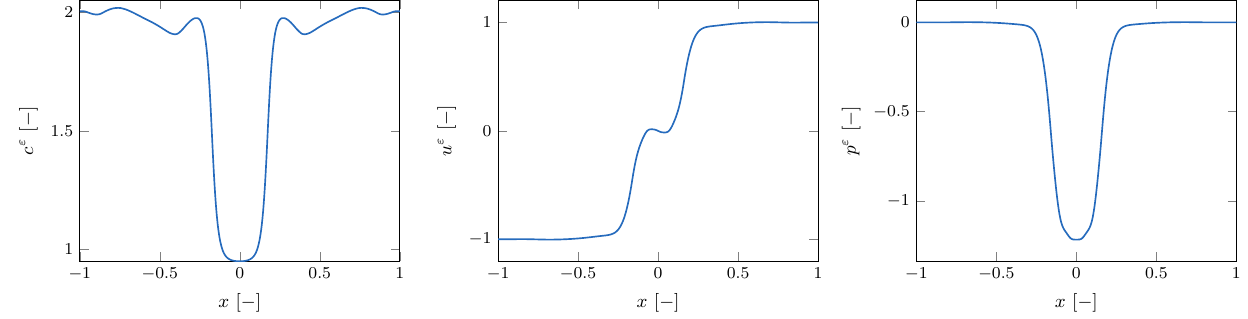}
    \caption{Evolution of the phase field $c^\varepsilon$, the velocity $u^\varepsilon$ and the pressure $p^\varepsilon$ at $t = 0.15$ for $\alpha = 0.5$, $\beta = 0.1$, $\gamma = 10^{-4}$ and $\delta = 0.0625$ for the \textit{expansion-like} setup.}
    \label{fig:ExpansionNSCH}
\end{figure}

%
\section{Conclusions and Outlook}
In this paper we investigated the diffuse-interface approach for incompressible two-phase flows based on the widely used  NSCH system
\eqref{NSCH}.  We    provided a novel approximative system  that  basically 
consists of a  first-order  sub-system, see \eqref{firstorderNSEK}. As our main results, we 
proved that the entire system is thermodynamically consistent and that the first-order sub-system  is --at least in one space dimension-- hyperbolic in the 
relevant state space.
For  exemplary cases we  could showcase a numerical
characteristic analysis. 
%
Finally, we presented   numerical results for the spatially one-dimensional  version \eqref{rewrittenfirstorderNSEK}. \\ 
This note contains partial results of  a  forthcoming paper which addresses more general NSCH models and the numerical solution of the relaxed friction-type system \eqref{firstorderNSEK} in multiple space dimensions \cite{KKR}. The major challenge in this regard is the development of an asymptotic-preserving scheme that employs an implicit-explicit time discretization to overcome the inherent stiffness of the three-parameter problem \eqref{firstorderNSEK}. The question of the convergence of solutions  $\bm{U}^\eps$ of the initial boundary value problem for \eqref{firstorderNSEK} to solutions $\bm U$ of the corresponding solution of the  initial boundary value problem for the limiting NSCH  model \eqref{NSCH} remains open. Possibly it can be done using the ideas for a two-parameter problem as in \cite{huang2024}. We only considered a simplified version of the Navier-Stokes-Cahn-Hilliard class. 
It would be interesting to explore whether $c$-dependent mobilities can be introduced, whether 
fluids with different densities and viscosities could be handled \cite{AbelsGG}, or even systems involving more than two phases \cite{Boyer, RovWo}. 

%
\bigskip

\textbf{Acknowledgement.}
This work was  supported by the German Research
Foundation (DFG), within the   International Research Training Group  GRK 2160 Droplet Interaction Technologies
and  the Collaborative Research Center on Interface-Driven Multi-Field Processes in Porous Media (SFB 1313, Project Number
327154368). Funding by the Deutsche Forschungsgemeinschaft (DFG, German Research Foundation) - SPP 2410 Hyperbolic Balance Laws in Fluid Mechanics: Complexity, Scales, Randomness (CoScaRa) is also  gratefully acknowledged. 

\bibliographystyle{abbrv} 
\bibliography{literature}{}

\end{document}